\documentclass[11pt,leqno,a4paper]{article}

\usepackage{amsthm}
\usepackage{mathrsfs}
\usepackage{amsmath}
\usepackage{amsfonts}
\usepackage{amssymb}
\usepackage[utf8]{inputenc}
\usepackage[T1]{fontenc}

\usepackage{graphicx}
\usepackage{xcolor, colortbl}
\usepackage{array,color}
\usepackage{arydshln}
\usepackage{threeparttable}
\usepackage{multirow}
\usepackage{float}
\usepackage{caption}
\usepackage{subcaption}
\usepackage{bm}
\usepackage{multirow}
\usepackage{bigdelim}

\usepackage[ruled,vlined]{algorithm2e}
\usepackage{algorithmic}
\usepackage{hhline}
\usepackage[square,numbers,sort&compress]{natbib}

\usepackage{upgreek}
\usepackage{tikz}
\usetikzlibrary{shadows}
\usetikzlibrary{shapes}

\newtheorem{rmk}{\textit{Remark}}[section]

\newtheorem{theorem}{Theorem}[section]
\newtheorem{proposition}[theorem]{Proposition}

\theoremstyle{remark}

\theoremstyle{definition}

\usepackage{authblk}

\title{Domain decomposition algorithms for the two dimensional nonlinear Schr{\"o}dinger equation and simulation of Bose-Einstein condensates}

\author[1]{Christophe Besse\thanks{christophe.besse@math.univ-toulouse.fr}}
\author[2]{Feng Xing\thanks{feng.xing@unice.fr}}

\affil[1]{Institut de Math{\'e}matiques de Toulouse UMR5219,
  Universit\'e de Toulouse; CNRS,
  UPS IMT, F-31062 Toulouse Cedex 9, France.}
\affil[2]{Laboratoire de Math\'ematiques J.A. Dieudonn\'e, UMR 7351 CNRS, University Nice Sophia Antipolis, team COFFEE, INRIA Sophia Antipolis M\'editerran\'ee, Parc Valrose 06108 Nice Cedex 02, France, and BRGM Orl\'eans France }

\begin{document}

\maketitle

\begin{abstract}

  In this paper, we apply the optimized Schwarz method to the two
  dimensional nonlinear Schr\"{o}dinger equation and extend this method
  to the simulation of Bose-Einstein condensates (Gross-Pitaevskii
  equation). We propose an extended version of the Schwartz method by
  introducing a
  preconditioned algorithm. The two algorithms are studied
  numerically. The experiments show that the preconditioned algorithm
  improves the convergence rate and reduces the computation time. In addition, the
  classical Robin condition and a newly constructed absorbing
  condition are used as transmission conditions. 

\end{abstract}


\textbf{Math. classification.} 35Q55; 65M55; 65Y05; 65M60.

\section{Introduction}

We are interested in solving the nonlinear Schr\"{o}dinger equation
and the Gross-Pitaevskii (GPE) equation by the optimized Schwarz
method. A large number of articles and books
\cite{Gander2012dd,Gander2008history, Gander2006osw} are devoted to
this method for different kinds of equations, for example the Poisson
equation \cite{Lions1990}, the Helmholtz equation
\cite{Boubendir2012,Gander2002helm} and the convection-diffusion
equation \cite{Nataf1994a}. Recently, the authors of
\cite{Halpern2010_sch,Antoine2014, XF20151d, XF20152dlin} applied the
domain decomposition method to the linear or nonlinear Schr\"{o}dinger
equation. More specificaly, in \cite{XF20151d, XF20152dlin}, the
authors proposed some newly efficient and scalable Schwarz methods for
1d or 2d linear Schr\"odinger equation and for 1d nonlinear
Schr\"odinger equation. These new algorithms could ensure high
scalability and reduce computation time. In this paper, we extend
these works to the two dimensional nonlinear case. 

The nonlinear Schr\"{o}dinger equation defined on a two dimensional bounded spatial domain $\Omega :=(x_l,x_r) \times (y_b,y_u)$, $x_l,x_r,y_b,y_u \in \mathbb{R}$ and $t\in(0,T)$ with general real potential $V(t,x,y) + f(\cdot)$ reads
\begin{equation}
  \label{Sch}
  \left
    \{\begin{array}{ll}
        i \partial_t u  + \Delta u + V(t,x,y) u + f(u)u = 0, \ (t,x,y) \in (0,T)\times \Omega, \\ 
        u(0,x,y) = u_0(x,y),
      \end{array} 
    \right.
  \end{equation}
  where $u_0 \in L^2(\Omega)$ is the initial datum. We complement the
  equation with homogeneous Neumann boundary condition on 
  bottom and top boundaries and Fourier-Robin boundary conditions
  on left and right boundaries. They read:
  \begin{equation}
    \label{BC}
    \partial_{\mathbf{n}} u= 0,\ (x,y)\in(x_l,x_r)\times\{y_b,y_u\},\quad \partial_{\mathbf{n}} u + S u = 0,\  (x,y)\in \{x_l,x_r\}\times(y_b,y_u),
  \end{equation}
  where $\partial_{\mathbf{n}}$ denotes the normal directive, $\mathbf{n}$ being the outwardly unit vector on the boundary $\partial \Omega$ (see Figure \ref{figomega}). The operator $S$ is given by
  \begin{align}
    & S u = -ip\cdot u, \ p \in \mathbb{R}^+, \label{BCRobin} \\
    \mathrm{or} \ & S u = -i \sqrt{i\partial_t + \Delta_{\Gamma} + V + f(u)} u, \label{BCS2p}
  \end{align}
  where $\Gamma=\{x_l,x_r\} \times (y_b,y_u)$. The Laplace–Beltrami
  operator $\Delta_{\Gamma}$ is $\partial_y^2$ in our case. 
  The operator $S$ in \eqref{BCS2p} is a pseudo differential operator
  constructed recently in \cite{Antoine2012} as an absorbing boundary
  operator, which is used to approximate the exact solution of the
  problem defined on $\mathbb{R}^2$, restricted to a bounded space
  domain.  
  
  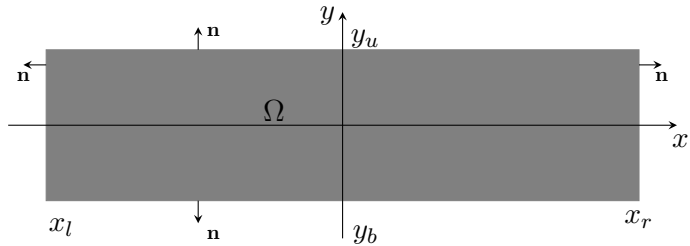
\begin{figure}[!htbp]
    \centering
    \begin{tikzpicture}
      \draw [color=gray, fill=gray] (0,0) rectangle (7.8,2);
      \draw[>=stealth,->] (-0.5,1) -- (8.3,1);
      \draw (8.1,0.8) node[right] {$x$};
      \draw [>=stealth,->] (3.9,-0.5) -- (3.9,2.5);
      \draw (3.7,2.2) node[above] {$y$};    
      
      \draw[>=stealth,->] (0,1.8) -- (-0.3,1.8);
      \draw (-0.3,1.8) node[below,scale=0.7] {$\mathbf{n}$};

      \draw[>=stealth,->] (7.8,1.8) -- (8.1,1.8);
      \draw (8.1,1.8) node[below,scale=0.7] {$\mathbf{n}$};    
      
      \draw[>=stealth,->] (2,2) -- (2,2.3);
      \draw (2.2,2.4) node[below,scale=0.7] {$\mathbf{n}$};      
      
      \draw[>=stealth,->] (2,0) -- (2,-0.3);
      \draw (2.2,-0.3) node[below,scale=0.7] {$\mathbf{n}$};       
      
      \draw (0.2,-0.1) node[below] {$x_l$};
      \draw (7.8,0) node[below] {$x_r$};
      \draw (4.2,-0.2) node[below] {$y_b$};
      \draw (4.2,2.4) node[below] {$y_u$};
      \draw (3.0,1.2) node[scale=0.8] {\Large $\Omega$};
    \end{tikzpicture}
    \caption{Spatial bounded domain $\Omega=(x_l,x_r) \times (y_b,y_u)$.}
    \label{figomega}
  \end{figure}

  Recently, the Schwarz
  algorithms have been applied to the one dimensional linear or nonlinear Schr\"{o}dinger
  equation \cite{XF20151d}. If the potential is linear and independent
  of time, then an interface problem allows to
  construct a global in space operator. It is possible to
  assemble it in parallel without too much computational efforts. Thanks
  to this operator, a new algorithm was introduced which is
  mathematically equivalent to the original Schwarz algorithm, but
  requires less iterations and 
  computation time. If the potential is general, the authors
  used a pre-constructed linear operator as preconditioner, which leads
  to a preconditioned algorithm. The preconditioner allows to reduce both the
  number of iterations and the computation time. These new algorithms
  have been extended to the two dimensional linear Schr\"{o}dinger
  equation \cite{XF20152dlin}, which also shows the effectiveness of the new algorithms. In this
  article, we propose to extend the results to the 2d nonlinear case and to
  the simulation of Bose-Einstein condensates. Following the naming in
  \cite{XF20151d}, we refer to the algorithm given by \eqref{algoglobal}
  as ``the classical algorithm'' and to the algorithm that will be presented in
  Section \ref{Sec_Palgo} as ``the preconditioned algorithm''. 

  The paper is organized as follows. We present in section \ref{Sec_cls} the the classical algorithm and some details about the discretization. A preconditioned algorithm is presented in section \ref{Sec_Palgo}. In section \ref{Sec_Imp}, we present the implementation of these algorithms on parallel computers. Numerical experiments are performed in section \ref{Sec_Num} and we focused on simulation of Bose-Einstein condensates. The last section draws a conclusion and suggests some future directions of research.

  \section{Classical algorithm}
  \label{Sec_cls}
  
  \subsection{Classical optimized Schwarz algorithm}	

  Let us discretize uniformly with $N_T$ intervals the time domain $(0,T)$. We define $\Delta t = T/N_T$ to be the time step. The usual semi-discrete in time 
  scheme developed by Durán and Sanz-Serna \cite{DUR2000} applied to \eqref{Sch} reads as
  \begin{displaymath}
    \begin{split}
      i\frac{u_{n} - u_{n-1}}{\Delta t} +  \Delta \frac{u_{n} + u_{n-1}}{2} + \frac{V_{n} + V_{n-1}}{2} \frac{u_{n} + u_{n-1}}{2} 
      +  f(\frac{u_{n} + u_{n-1}}{2}) \frac{u_{n} + u_{n-1}}{2} = 0, & \quad 1 \leqslant n \leqslant N_T
    \end{split}
  \end{displaymath} 
  where $u_{n}(x,y), (x,y) \in \Omega$ denotes the approximation of the solution $u(t_n,x,y)$ to the Schr\"{o}dinger equation \eqref{Sch} at time $t_n = n \Delta t$ and $V_{n}(x,y)=V(t_n,x,y)$. By introducing new variables
  $v_{n} = (u_{n} + u_{n-1})/2$ with $v_{0} = u_{0}$ and $W_{n}=(V_{n}+V_{n-1})/2$, this scheme can be written as
  \begin{equation}
    \label{CNS}
    \mathscr{L}_{\mathbf{x}} v_{n} = 2i\frac{u_{n-1}}{\Delta t},
  \end{equation}
  where the operator $\mathscr{L}_{\mathbf{x}}$ is defined by
  \begin{displaymath}
    \mathscr{L}_{\mathbf{x}} v_{n} := \frac{2i}{\Delta t} v_{n} + \Delta v_{n} + W_{n} v_{n} + f(v_{n}) v_{n}.
  \end{displaymath}
  For any $1\leq n \leq N_T$, the equation \eqref{CNS} is stationary. We can therefore apply the optimized Schwarz method. Let us decompose the spatial domain $\Omega$ into $N$ subdomains $\Omega_j =(a_j,b_j),j=1,2,...,N$ without overlap as shown in Figure \ref{figsub3} for $N=3$. The Schwarz algorithm is an iterative process and we identify the iteration number thanks to label $k$. We denote by $v_{j,n}^k$ the solution on subdomain $\Omega_j$ at iteration $k=1,2,...$ of the Schwarz algorithm (resp $u_{j,n}^k$). Assuming that $u_{0,n-1}$ is known, the optimized Schwarz algorithm for \eqref{CNS} consists in applying the following sequence of iterations for $j=2,3,...,N-1$
  \begin{equation}
    \label{algoglobal}
    \left \{
      \begin{array}{l}
        \displaystyle   \mathscr{L}_{\mathbf{x}}  v_{j,n}^{k} = \frac{2i}{\Delta t} u_{j,n-1}, \ (x,y) \in \Omega_j,  \\[2mm]
        \partial_{\mathbf{n}_j} v^{k}_{j,n} + \overline{S}_j v^{k}_{j,n} = \partial_{\mathbf{n}_{j}} v^{k-1}_{j-1,n} + \overline{S}_j v^{k-1}_{j-1,n} , \ x=a_j,\ y \in (y_b,y_u),\\[2mm]
        \partial_{\mathbf{n}_j} v^{k}_{j,n} + \overline{S}_j v^{k}_{j,n} = \partial_{\mathbf{n}_{j}} v^{k-1}_{j+1,n} + \overline{S}_j v^{k-1}_{j+1,n}, \ x=b_j, \ y \in (y_b,y_u),
      \end{array} \right.
  \end{equation}
  with a special treatment for the two extreme subdomains
  $\Omega_1$ and $\Omega_N$ since
  the boundary conditions are imposed on $\{x=a_1\}\times(y_b,y_u)$ and $\{x=b_N\}\times(y_b,y_u)$ 
  \begin{displaymath}
    \partial_{\mathbf{n}_1} v^{k}_{1,n} + \overline{S}_j v^{k}_{1,n} = 0, x=a_j, \quad 
    \partial_{\mathbf{n}_N} v^{k}_{N,n} + \overline{S}_j v^{k}_{N,n} = 0, \ x=b_N.
  \end{displaymath} 

          %
  \definecolor{col1}{rgb}{0.8,1,0.8}
  \definecolor{col2}{rgb}{1,0.8,0.8}
  \definecolor{col3}{rgb}{0.8,0.8,1}
  \begin{figure}[!htbp]
    \centering
    \begin{tikzpicture}
      \draw [color=gray,fill=col1] (0,0) rectangle (2.6,2);
      \draw [color=gray,fill=col2] (2.6,0) rectangle (5.2,2);
      \draw [color=gray,fill=col3] (5.2,0) rectangle (7.8,2);
      \draw[>=stealth,->] (-0.5,1) -- (8.3,1);
      \draw (8.1,0.8) node[right] {$x$};
      \draw [>=stealth,->] (3.9,-0.5) -- (3.9,2.5);
      \draw (3.7,2.2) node[above] {$y$};    
      
      \draw[>=stealth,->] (2.6,1.8) -- (2.3,1.8);
      \draw (2.2,1.8) node[below,scale=0.7] {$\mathbf{n}_2$};

      \draw[>=stealth,->] (5.2,1.8) -- (5.5,1.8);
      \draw (5.6,1.8) node[below,scale=0.7] {$\mathbf{n}_2$};    
      
      \draw (0.2,-0.1) node[below] {$x_l=a_1$};
      \draw (2.6,0) node[below] {$b_1=a_2$};
      \draw (5.2,0) node[below] {$b_2=a_3$};
      \draw (7.8,0) node[below] {$b_2=x_r$};
      \draw (1.0,1.3) node[scale=0.8] {$\Omega_1$};
      \draw (3.6,1.3) node[scale=0.8] {$\Omega_2$};
      \draw (6.1,1.3) node[scale=0.8] {$\Omega_3$};
    \end{tikzpicture}
    \caption{Domain decomposition without overlap, $N=3$.}
    \label{figsub3}
  \end{figure}
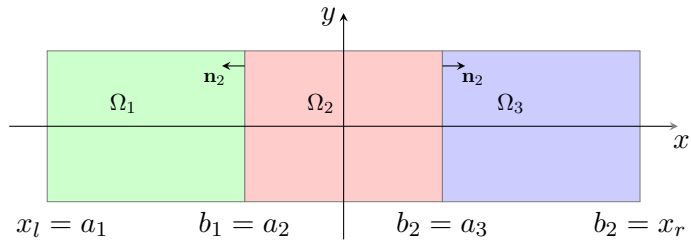

  Various transmission operator $\overline{S}$ can be considered. The first one is the classical widely used Robin transmission condition
  \begin{equation}
    \label{TCSRobin}
    \mathrm{Robin:} \quad \overline{S}_j v = -ip \cdot v, \ p \in \mathbb{R}^+.
  \end{equation}
  Traditionally, the optimal transmission operator is
  given in term of transparent boundary conditions (TBCs). For the
  nonlinear two dimensional Schr\"odinger equation, we only have access to
  approximated version of the TBCs given by the recently
  constructed absorbing boundary condition
  $S_{\mathrm{pade}}^m$ \cite{Antoine2012,Antoine2013}
  which we used as the transmission condition
  \begin{equation}
    \label{TCS}
    S_{\mathrm{pade}}^m: \quad \overline{S}_j v = -i \sum_{s=0}^m a_s^m v + i\sum_{s=1}^m a_s^m d_s^m \varphi_{j,s},  \ x=a_j,b_j.
  \end{equation}
  The operator $S_{\mathrm{pade}}^m$ \cite{Antoine2012, Antoine2013} is originally constructed by using some pseudo differential techniques. Numerically it is 
  approximated by Pad\'e approximation of order $m$
  \begin{displaymath}
    S_{\mathrm{pade}}^m v = -i \sqrt{\frac{2i}{\Delta t} + \Delta_{\Gamma_j} + W + f(v) } v \approx
    \Big( -i\sum_{s=0}^m a_s^m + i\sum_{s=1}^m a_s^m d_s^m (\frac{2i}{\Delta t} + \Delta_{\Gamma_j} + W + f(v) + d_s^m)^{-1} \Big) v,
  \end{displaymath}
  where $\Gamma_j=\{x=a_j,b_j\} \times (y_b,y_u)$. The Laplace-Beltrami operator $\Delta_{\Gamma_j}$ is $\partial_{yy}$ in our case and the constant coefficients are $a_s^m = e^{i\theta/2}/\big(m\cos^2(\frac{(2s-1)\pi}{4m}) \big)$, $d_s^m = e^{i\theta} \tan^2 (\frac{(2s-1)\pi}{4m})$, $s=0,1,...,m$, $\theta = \frac{\pi}{4}$.  
  The auxiliary functions $\varphi_{j,s}$, $j=1,2,...,N$,$ s=1,2,...,m$ are defined as solution of the set of equations 
  \begin{equation}
  \label{TCS_phi}
    \Big( \frac{2i}{\Delta t} + \Delta_{\Gamma_j} + W + f(v) + d_s^m \Big) \varphi_{j,s}(x,y) = v,  \ (x,y) \in (a_j,b_j)\times(y_b,y_u).
  \end{equation}
  
  Let us introduce the fluxes $l_{j,n}^k$ and $r_{j,n}^k$ defined as
  \begin{displaymath}
    \begin{split}
      l_{j,n}^k(y) & = \partial_{\mathbf{n}_j} v^{k}_{j,n}(a_j,y) + \overline{S}_j v^{k}_{j,n} (a_j,y), \ y \in (y_b,y_u),\\
      r_{j,n}^k(y) & = \partial_{\mathbf{n}_j} v^{k}_{j,n}(b_j,y) + \overline{S}_j v^{k}_{j,n} (b_j,y), \ y \in (y_b,y_u),
    \end{split}
  \end{displaymath}
  with a special definition for the two extreme subdomains:
  $l_{1,n}^k = r_{N,n}^k=0$. Thus, the algorithm \eqref{algoglobal}
  can be splitted into local problems on subdomains
  $\Omega_j,j=1,2,...,N$  
  \begin{equation}
    \label{algolocal}
    \left \{
      \begin{array}{l}
        \displaystyle   \mathscr{L}_{\mathbf{x}}  v_{j,n}^{k} = \frac{2i}{\Delta t} u_{j,n-1}, \\[2mm]
        \partial_{\mathbf{n}_j} v^{k}_{j,n} + \overline{S}_j v^{k}_{j,n} = l_{j,n}^k, \ x=a_j,\\[2mm]
        \partial_{\mathbf{n}_j} v^{k}_{j,n} + \overline{S}_j v^{k}_{j,n} = r_{j,n}^k, \ x=b_j,
      \end{array} \right.
  \end{equation}
  and flux problems 
  \begin{equation}
    \label{fluxlr}
    \left \{
      \begin{array}{l}
        l_{j,n}^{k+1} = -r_{j-1,n}^{k} + 2 \overline{S}_j v_{j-1,n}^k(b_{j-1},y), \ j=2,3,...,N,\\[2mm]
        r_{j,n}^{k+1} = -l_{j+1,n}^{k} + 2 \overline{S}_j v_{j+1,n}^k(a_{j+1},y),\ j=1,2,...,N-1.
      \end{array} \right. 
  \end{equation}

  \subsection{Preliminaries related to space discretization}    
  
  Without loss of generality, we present the space discretization of
  the semi-discrete Schr\"odinger equation defined on the bounded
  domain $(a,b) \times (y_b,y_u), a,b\in \mathbb{R}$ 
  \begin{equation}
    \label{eqdist}
    \left \{
      \begin{array}{l}
        \displaystyle \frac{2i}{\Delta t} v + \Delta v + W(x,y) v + f(v) v = \frac{2i}{\Delta t} h(x,y),  \\[2mm]
        \partial_{\mathbf{n}} v + \overline{S} v = l(x,y) , \ x=a,\ y \in (y_b,y_u),\\[2mm]
        \partial_{\mathbf{n}} v + \overline{S} v = r(x,y), \ x=b, \ y \in (y_b,y_u),\\[2mm]
        \partial_{\mathbf{n}} v = 0, \ y=y_b, \ y=y_u,
      \end{array} \right.
  \end{equation}
  where $W(x,y)v+f(v)v$ plays the role of the semi-discrete potential in
  \eqref{CNS} and $l(x,y)$, $r(x,y)$ are two functions. The operator
  $\overline{S}$ is Robin or $S_{\mathrm{pade}}^m$ given respectively by
  \eqref{TCSRobin} and \eqref{TCS}.
                                 %
                                 %

  If $f\neq 0$, then the system \eqref{eqdist} is nonlinear. The computation of $v$ is accomplished by a fixed point procedure. If we consider the Robin transmission condition, we take $\zeta^{0}=h$ and compute the solution $v$ as the limit of the iterative procedure with respect to the label $q$, $q=1,2,...$
  \begin{equation}
    \label{localpbrobin}
    \left \{
      \begin{array}{l}
        \displaystyle  \Big( \frac{2i}{\Delta t} + \Delta + W + f(\zeta^{q-1}) \Big) \zeta^q = \frac{2i}{\Delta t} h, \\[2mm]
        \partial_{\mathbf{n}} \zeta^q - ip \cdot \zeta^q = l, \ x=a,\\[2mm]
        \partial_{\mathbf{n}} \zeta^q -ip \cdot \zeta^q = r, \ x=b.
      \end{array} \right.
  \end{equation}
  For the transmission condition $S_{\mathrm{pade}}^m$, we take
  $\zeta^{0}=h$ and $\phi^{0}_{s}=0$, $s=1,2,...,m$. The
  unknowns $v$ and $\varphi_{s},s=1,2,...,m$ are computed as the
  limit (with respect to $q$) of $\zeta^{q}$ and $\phi^{q}_{s}$,
  $s=1,2,\cdots,m$, which are solutions
  of the following coupled system 
  \begin{equation}
    \label{localpbS2p}
    \left \{
      \begin{array}{l}
        \displaystyle  \big( \frac{2i}{\Delta t} + \Delta + W_{n} + f(\zeta^{q-1} \big) \zeta^q = \frac{2i}{\Delta t} h, \\[2mm]
        \displaystyle \big( \frac{2i}{\Delta t} + W + \Delta_{\Gamma} + d_s^m \big) \phi^{q}_{s} = \zeta^{q} - f(\zeta^{q-1} ) \phi^{q-1}_{s}, \\
        \displaystyle   \partial_{\mathbf{n}} \zeta^q -i \sum_{s=0}^m a_s^m \zeta^{q} + i\sum_{s=1}^m a_s^m d_s^m \phi^{q}_{s} = l, \ x=a,\\[2mm]
        \displaystyle  \partial_{\mathbf{n}} \zeta^q -i \sum_{s=0}^m a_s^m \zeta^{q} + i\sum_{s=1}^m a_s^m d_s^m \phi^{q}_{s} = r, \ x=b.
      \end{array} \right.
  \end{equation}
          %
  
  The spatial approximation is realized with the standard $\mathbb{Q}_1$
  finite element method with uniform mesh. The mesh size of a discrete
  element is $(\Delta x, \Delta y)$. We denote by $N_x$ (resp. $N_y$)
  the number of nodes in $x$ (resp. $y$) direction on each
  subdomain. Let us denote by $\mathbf{v}$ (resp. $\mathbf{h}$) the
  nodal interpolation vector of $v$ (resp. $h$), $\bm{\zeta}^{q}$ the
  nodal interpolation vector of $\zeta^{q}$, $\mathbf{l}$
  (resp. $\mathbf{r}$) the nodal interpolation vector of $l(x,y)$
  (resp. $r(x,y)$), $\mathbb{M}$ the mass matrix, $\mathbb{S}$ the
  stiffness matrix and $\mathbb{M}_{W}$ the generalized mass matrix with
  respect to $\int_{\Omega} W v \phi dx$. Let $\mathbb{M}^{\Gamma}$ the
  boundary mass matrix, $\mathbb{S}^{\Gamma}$ the boundary stiffness
  matrix and $\mathbb{M}^{\Gamma}_{W}$ (resp. $\mathbb{M}^{\Gamma}_{W}$)
  the generalized boundary mass matrix with respect to $\int_{\Gamma} W
  v \phi d\Gamma$ (resp. $\int_{\Gamma} f(v) \phi d\Gamma$). We denote
  by $Q_{l}$ (resp. $Q_{r}$) the restriction operators (matrix) from
  $\Omega$ to $\{a\} \times (y_b,y_u)$ (resp. $\{b\} \times (y_b,y_u)$)
  and $Q^{\top}=(Q_{l}^{\top}, Q_{r}^{\top})$ where ``$\cdot^\top$'' is
  the standard notation of the transpose of a matrix or a vector. The
  matrix formulation for \eqref{localpbrobin} is therefore given by 
  \begin{equation}
    \label{NproblemRobinNL}
    \quad \Big( \mathbb{A} + ip \cdot \mathbb{M}^{\Gamma} + \mathbb{M}^{\Gamma}_{f(\bm{\zeta}^{q-1})}\Big)
    \bm{\zeta}^{q} =  \frac{2i}{\Delta t} \mathbb{M}
    \mathbf{h} - \mathbb{M}^{\Gamma}  Q^{\top} 
    \begin{pmatrix}
      \mathbf{l}\\
      \mathbf{r}
    \end{pmatrix},   
  \end{equation}
  where $\mathbb{A} = \frac{2i}{\Delta t} \mathbb{M} - \mathbb{S} + \mathbb{M}_{W}$. The size of this linear system is $N_x \times N_y$. If we consider the transmission condition $S_{\mathrm{pade}}^m$, we have
  \begin{gather}
    \label{NproblemSPNL}
    \mathcal{A} \begin{pmatrix}
      \bm{\zeta}^{q}\\
      \bm{\phi}^{q}_{1}\\
      \bm{\phi}^{q}_{2}\\
      \vdots\\
      \bm{\phi}^{q}_{m}
    \end{pmatrix} 
    :=
    \begin{pmatrix}
      \mathbb{A} + i(\sum_{s=1}^m a_s^m)  \cdot
      \mathbb{M}^{\Gamma} & \mathbb{B}_{1} & \mathbb{B}_{2} & \cdots & \mathbb{B}_{m} \\
      \mathbb{C} & \mathbb{D}_{1} \\
      \mathbb{C} & & \mathbb{D}_{2} \\
      \vdots & & & \ddots \\
      \mathbb{C} & &  & & \mathbb{D}_{m}
    \end{pmatrix}
    \begin{pmatrix}
      \bm{\zeta}^{q}\\
      \bm{\phi}^{q}_{1}\\
      \bm{\phi}^{q}_{2}\\
      \vdots\\
      \bm{\phi}^{q}_{m}
    \end{pmatrix}  \nonumber  \\
    = \frac{2i}{\Delta t}
    \begin{pmatrix}
      \mathbb{M} \mathbf{h}\\
      \mathbf{0} \\
      \mathbf{0} \\
      \vdots\\ 
      \mathbf{0}       
    \end{pmatrix}
    - 
    \begin{pmatrix}
      \mathbb{M}_{f(\bm{\zeta^{q-1}})} \bm{\zeta^{q-1}}\\
      \mathbb{M}_{f(\bm{\zeta^{q-1}})}^{\Gamma} \bm{\phi}^{q-1}_{1}\\
      \vdots\\
      \mathbb{M}_{f(\bm{\zeta_j^{q-1}})}^{\Gamma} \bm{\phi}^{q-1}_{m}\\
    \end{pmatrix}
    -
    \begin{pmatrix}
      \mathbb{M}^{\Gamma} \cdot Q^{\top}
      \begin{pmatrix}
        \mathbf{l}\\
        \mathbf{r}
      \end{pmatrix}\\
      0\\
      \vdots \\
      0
    \end{pmatrix}.
  \end{gather}                 
  with
  \begin{align*}
    & \mathbb{B}_{s} = -i a_s^m d_s^m \mathbb{M}^{\Gamma} Q^{\top} , \ 1
      \leqslant s \leqslant m,\\
    & \mathbb{C} = - Q \mathbb{M}^{\Gamma}, \\
    & \mathbb{D}_{s} = Q (\frac{2i}{\Delta
      t} \mathbb{M}^{\Gamma} -  \mathbb{S}^{\Gamma} +
      \mathbb{M}^{\Gamma}_{W} + d_s^m \mathbb{M}^{\Gamma}) Q^{\top},  \ 1
      \leqslant s \leqslant m.
  \end{align*}      
  It is a linear system with unknown $(\bm{\zeta}^q, \bm{\phi}^{q}_{1}, ..., \bm{\phi}^{q}_{m})$ where $\bm{\phi}^{q}_{s}$ is the nodal interpolation of $\phi^{q}_{s}$ on the boundary and $\bm{\varphi}_{s}$ is the nodal interpolation of $\varphi_{s}$. The vectors $\mathbf{v}$ and $\bm{\varphi}_{s}$ are computed by
  \begin{equation*}
    \mathbf{v} = \lim_{q \rightarrow \infty} \bm{\zeta}^{q}, \quad
    \bm{\varphi}_{s} = \lim_{q \rightarrow \infty} {\bm{\phi}}^{q}_{s}, \ s=1,2,...,m.
  \end{equation*}   
  In addition, the discrete form of the transmission operator $\overline{S}$ is given by
  \begin{equation}
    \label{TCSdist}
    \begin{split}
      & \mathrm{Robin:} \quad \mathbf{S} \mathbf{v} = -ip \cdot \mathbf{v}, \ p \in \mathbb{R}^+, \\
      & S_{\mathrm{pade}}^m: \quad \mathbf{S} \mathbf{v} = -i \sum_{s=0}^m a_s^m \mathbf{v} + i\sum_{s=1}^m a_s^m d_s^m \bm{\varphi}_{s}.
    \end{split}
  \end{equation}   
  
  \begin{rmk}
    The $S_{\mathrm{pade}}^m$ transmission condition involves a larger linear system to solve than the one of the Robin transmission condition. The cost of the algorithm with the $S_{\mathrm{pade}}^m$ transmission condition is therefore more expensive.
  \end{rmk}   
  
  If the potential is linear $f=0$, then from the system \eqref{eqdist} we have
  \begin{equation}
    \label{NproblemRobinL}
    \mathrm{Robin:} \quad\Big( \mathbb{A} + ip \cdot \mathbb{M}^{\Gamma} \Big)
    \mathbf{v} =  \frac{2i}{\Delta t} \mathbb{M}
    \mathbf{h} - \mathbb{M}^{\Gamma}  Q^{\top} 
    \begin{pmatrix}
      \mathbf{l}\\
      \mathbf{r}
    \end{pmatrix},   
  \end{equation}
  and
  \begin{gather}
    \label{NproblemSPL}
    S_{\mathrm{pade}:} \quad
    \mathcal{A} \begin{pmatrix}
      \mathbf{v}\\
      \bm{\varphi}_{1}\\
      \bm{\varphi}_{2}\\
      \vdots\\
      \bm{\varphi}_{m}
    \end{pmatrix} 
    = \frac{2i}{\Delta t}
    \begin{pmatrix}
      \mathbb{M} \mathbf{h}\\
      \mathbf{0} \\
      \mathbf{0} \\
      \vdots\\ 
      \mathbf{0}       
    \end{pmatrix}
    - 
    \begin{pmatrix}
      \mathbb{M}^{\Gamma} \cdot Q^{\top}
      \begin{pmatrix}
        \mathbf{l}\\
        \mathbf{r}
      \end{pmatrix}\\
      0\\
      \vdots \\
      0
    \end{pmatrix}.
  \end{gather}   
  Directly from the definition of $\mathcal{A}$, \eqref{NproblemSPL} can
  be written as one equation for $\mathbf{v}$ 
  \begin{equation}
    \Big( \mathbb{A} + i(\sum_{s=1}^m a_s^m)  \cdot \mathbb{M}^{\Gamma}
    - \sum_{s=1} \mathbb{B}_s \mathbb{D}_s^{-1} \mathbb{C}_s \Big) \mathbf{v} 
    =
    \frac{2i}{\Delta t} \mathbb{M}
    \mathbf{h} - \mathbb{M}^{\Gamma}  Q^{\top} 
    \begin{pmatrix}
      \mathbf{l}\\
      \mathbf{r}
    \end{pmatrix}.  
  \end{equation}  
  Note that numerically, we implement \eqref{NproblemRobinL} and \eqref{NproblemSPL} to compute $\mathbf{v}$ (and $\bm{\varphi}_{s}$, $s=1,2,...,m$).

  \subsection{Classical discrete algorithm}
  
  Following what is done for the complete domain $\Omega \times [0,T]$,
  we discretize the equations \eqref{algolocal} and \eqref{fluxlr} on each subdomain $\Omega_j$ at each time step $n=1,2,...,N_T$. Accordingly, on each subdomain $\Omega_j$, let us denote by
  \begin{itemize}
  \item $\mathbb{A}_{j,n} = \frac{2i}{\Delta t}\mathbb{M}_{j} - \mathbb{S}_{j} + \mathbb{M}_{j,W_n}$ where $\mathbb{M}_{j}$ is the mass matrix, $\mathbb{S}_{j}$ is the stiffness matrix, $\mathbb{M}_{j,W_n}$ is the generalized mass matrix with respect to $\int_{\Omega_j} W_n v \phi dx$,
  \item $\mathbb{M}^{\Gamma_j}_{W_n}$ the generalized boundary mass matrix with respect to $\int_{\Gamma_j} W_n v \phi d\Gamma$,
    $\mathbb{M}^{\Gamma_j}_{f}$ the generalized boundary mass matrix with respect to $\int_{\Gamma_j} f(v) \phi d\Gamma$ where $\Gamma_j=\{x=a_j,b_j\}\times(y_b,y_u)$, 
  \item $Q_{j,l}$ and $Q_{j,r}$ the restriction operators (matrix) from $\Omega_j$ to its boundary $\{a_j\} \times (y_b,y_u)$ and $\{b_j\} \times (y_b,y_u)$ respectively, 
    $Q_j^{\top}=(Q_{j,l}^{\top}, Q_{j,r}^{\top})$,
  \item $\mathbb{B}_{j,s}$, $\mathbb{C}_{j}$, $\mathbb{D}_{j,n,s}$ the matrix associated with the operator $S_{\mathrm{pade}}^m$,
  \item $\mathbf{v}_{j,n}^k$ (resp. $\mathbf{u}_{j,n}^k$) the interpolation vectors of ${v}_{j,n}^k$ (resp. ${u}_{j,n}^k$).
  \end{itemize}   
  We denote by $\mathbf{l}_{j,n}^k$ (resp. $\mathbf{r}_{j,n}^k$) the nodal interpolation vector of $l_{j,n}^k$ (resp. $r_{j,n}^k$).  The classical algorithm is initialized by an initial
  guess of $\mathbf{l}_{j,n}^0$ and $\mathbf{r}_{j,n}^0$, $j=1,2,...,N$. The
  boundary conditions for any subdomain $\Omega_j$ at
  iteration $k+1$ involve the knowledge of the values of
  the functions on adjacent subdomains $\Omega_{j-1}$ and
  $\Omega_{j+1}$ at prior iteration $k$. Thanks to the
  initial guess, we can \emph{solve} the Schr\"odinger equation on
  each subdomain, allowing to build the new boundary
  conditions for the next step, \emph{communicating} them to other
  subdomains. This procedure is summarized 
  in \eqref{AlgoGraph} for $N=3$ subdomains at iteration
  $k$.
  \begin{equation}
    \label{AlgoGraph}
    \begin{pmatrix}
      \mathbf{r}_{1,n}^k \\
      \mathbf{l}_{2,n}^k \\
      \mathbf{r}_{2,n}^k \\
      \mathbf{r}_{3,n}^k
    \end{pmatrix}
    \underrightarrow{\hspace{0.1cm} \text{\footnotesize{\emph{Solve}}} \hspace{0.1cm}}
    \begin{pmatrix}
      \mathbf{v}_{1,n}^k \\
      \mathbf{v}_{2,n}^k \\
      \mathbf{v}_{3,n}^k \\
    \end{pmatrix}
    \underrightarrow{ }
    \begin{pmatrix}
      - \mathbf{r}_{1,n}^k + 2 \mathbf{S} ({Q}_{1,r} \mathbf{v}_{1,n}^k) \\
      - \mathbf{l}_{2,n}^k + 2 \mathbf{S} ({Q}_{2,l} \mathbf{v}_{j,n}^k) \\
      - \mathbf{r}_{2,n}^k + 2 \mathbf{S} ({Q}_{2,r} \mathbf{v}_{j,n}^k) \\
      - \mathbf{l}_{3,n}^k + 2 \mathbf{S} ({Q}_{3,l} \mathbf{v}_{N,n}^k)
    \end{pmatrix}
    \underrightarrow{ \hspace{0.1cm} \text{\footnotesize{\emph{Comm.}}} \hspace{0.1cm} }
    \begin{pmatrix}
      \mathbf{r}_{1,n}^{k+1} \\
      \mathbf{l}_{2,n}^{k+1} \\
      \mathbf{r}_{2,n}^{k+1} \\
      \mathbf{l}_{3,n}^{k+1} \\
    \end{pmatrix}.
  \end{equation}    
  Let us define the discrete interface vector by 
  \begin{displaymath}
    \mathbf{g}_n^{k,\top}=(\mathbf{r}_{1,n}^{k,\top},\cdots,\mathbf{l}_{j,n}^{k,\top},\mathbf{r}_{j,n}^{k,\top},\cdots,\mathbf{l}_{N,n}^{k,\top}).
  \end{displaymath}
  Thanks to this definition, we give a new interpretation to the algorithm which can be written as
  \begin{equation}
    \label{algo_d}
    \mathbf{g}_n^{k+1} = \mathcal{R}_{h,n} \mathbf{g}_n^k = I - (I - \mathcal{R}_{h,n}) \mathbf{g}_n^k.
  \end{equation} 
  where $I$ is identity operator and $\mathcal{R}_{h,n}$ is an operator. The solution to this iteration process is given as the
  solution to the discrete interface problem
  \begin{equation*}
    \label{interfacepb_d}
    ( I - \mathcal{R}_{h,n}) \mathbf{g}_n = 0.
  \end{equation*}  
  
  \begin{rmk}
    If $f=0$, the discretization of \eqref{algolocal} on each subdomain is
    \begin{gather}
      \mathrm{Robin:} \quad\Big( \mathbb{A}_{j,n} + ip \cdot \mathbb{M}^{\Gamma_j} \Big)
      \mathbf{v}_{j,n}^k =  \frac{2i}{\Delta t} \mathbb{M}
      \mathbf{u}_{j,n-1}^k - \mathbb{M}^{\Gamma_j}  Q^{\top} 
      \begin{pmatrix}
        \mathbf{l}_{j,n}^k\\
        \mathbf{r}_{j,n}^k
      \end{pmatrix},  \label{distrobin} \\
      S_{\mathrm{pade}:} \quad \Big( \mathbb{A}_{j,n} + i(\sum_{s=1}^m a_s^m)  \cdot \mathbb{M}^{\Gamma_j}
      - \sum_{s=1}^m \mathbb{B}_{j,s} \mathbb{D}_{j,n,s}^{-1} \mathbb{C}_{j} \Big) \mathbf{v}_{j,n}^k
      =
      \frac{2i}{\Delta t} \mathbb{M}
      \mathbf{u}_{j,n-1}^k - \mathbb{M}^{\Gamma_j}  Q^{\top} 
      \begin{pmatrix}
        \mathbf{l}_{j,n}^k\\
        \mathbf{r}_{j,n}^k
      \end{pmatrix}.  \label{dists2p}
    \end{gather}  
  \end{rmk}

  \section{Preconditioned algorithm}
  \label{Sec_Palgo}
  The application of the nonlinear operator $\mathcal{R}_{h,n}$ to $\mathbf{g}^k_n$ is expensive. In this section, we propose to add a preconditioner $P^{-1}$ in \eqref{algo_d}, which leads to a preconditioned algorithm
  \begin{equation}
    \label{algo_pd}
    \mathbf{g}_n^{k+1} = I - P^{-1}(I - \mathcal{R}_{h,n})\mathbf{g}_n^k.
  \end{equation}
  Here $P$ is a non singular matrix. To defined it,    
  let us consider the free Schr\"odinger equation with a zero potential $V=0$, $f=0$. We show in Propositions \ref{proprobin} and \ref{props2p} that in this case, the operator $\mathcal{R}_{h,n}$ is linear
  \begin{displaymath}
    \mathcal{R}_{h,n} \mathbf{g}_n^k = \mathcal{L}_h \mathbf{g}_n^k + \mathbf{d}_n,
  \end{displaymath}
  where $\mathcal{L}_h$ is a block matrix as defined by \eqref{Lmpi}
  and $\mathbf{d}_n$ is a vector (the notation ``MPI $j$'' is used
  in the next section). The matrix $\mathcal{L}_h$ is
  independent of the time step $n$. The size of each block $X^{k,l}$
  is $N_y
  \times N_y$. 
          %
  \begin{equation}
    \label{Lmpi}
    \mathcal{L}_{h} = 
    \begin{pmatrix}
      \multicolumn{1}{l}{\overbrace{\hspace{2.0em}}^{\mathrm{MPI}\ 0}} & 
      \multicolumn{2}{l}{\overbrace{\hspace{5.0em}}^{\mathrm{MPI}\ 1}} &
      \multicolumn{2}{l}{\overbrace{\hspace{5.0em}}^{\mathrm{MPI}\ 2}} &
      & 
      \multicolumn{2}{l}{\overbrace{\hspace{8.0em}}^{\mathrm{MPI}\ N-2}} &
      \multicolumn{2}{l}{\overbrace{\hspace{2.0em}}^{\mathrm{MPI}\ N-1}}
      \\
      & X^{2,1} & X^{2,2} & & & \\
      X^{1,4} \\
      & & & X^{3,1} & X^{3,2} \\
      & X^{2,3} & X^{2,4} \\
      & & & & & \cdots \\
      & & & X^{3,3} & X^{3,4} \\
      & & & & & & X^{N-1,1} & X^{N-1,2}\\
      & & & & &\cdots \\
      & & & & & & & & X^{N,1}\\
      & & & & & & X^{N-1,3} & X^{N-1,4}
    \end{pmatrix}.
  \end{equation}
  Thus, we propose here
  \begin{displaymath}
    P = I - \mathcal{L}_{h}.
  \end{displaymath}
  Note that since $\mathcal{L}_h$ is independent of time step $n$, the preconditioner is constructed once and used for all time steps.

  \begin{proposition}
    \label{proprobin}
    For the Robin transmission condition, assuming that $V=0$ and $f=0$, the matrix $\mathcal{L}_h$ takes the form \eqref{Lmpi} and $\mathcal{L}_h$ is independent of time step $n$. In addition, if the subdomains are equal, then
    \begin{equation}
      \label{blockeq}
      \begin{aligned}
        & X^{2,1} = X^{3,1} = \cdots = X^{N,1}, \ X^{2,2} = X^{3,2} = \cdots = X^{N-1,2},\\
        & X^{2,3} = X^{3,3} = \cdots = X^{N-1,3}, \ X^{1,4} = X^{2,4} = \cdots =   X^{N-1,4}.
      \end{aligned}
    \end{equation}
  \end{proposition}                            
  \begin{proof}
    First, by some straight forward calculations using \eqref{distrobin} and \eqref{AlgoGraph}, we can verify that
    \begin{equation}
      \label{matXRobin}
      \begin{aligned}
        & X^{j,1} =  -I - 2ip \cdot {Q}_{j,l} (\mathbb{A}_{j,n}
        +ip \cdot \mathbb{M}^{\Gamma_j})^{-1}
        \mathbb{M}^{\Gamma_{j}} {Q}_{j,l}^{\top}, \\
        & X^{j,2} =  - 2ip \cdot {Q}_{j,l} (\mathbb{A}_{j,n}
        +ip \cdot \mathbb{M}^{\Gamma_j})^{-1}
        \mathbb{M}^{\Gamma_{j}} {Q}_{j,r}^{\top}, \\
        & X^{j,3} =  - 2ip \cdot {Q}_{j,r} (\mathbb{A}_{j,n}
        +ip \cdot \mathbb{M}^{\Gamma_j})^{-1}
        \mathbb{M}^{\Gamma_{j}} {Q}_{j,l}^{\top}, \\
        & X^{j,4} =  -I - 2ip \cdot {Q}_{j,r} (\mathbb{A}_{j,n}
        +ip \cdot \mathbb{M}^{\Gamma_j})^{-1}
        \mathbb{M}^{\Gamma_{j}} {Q}_{j,r}^{\top},
      \end{aligned}
    \end{equation}  
    and $\mathbf{d}_{n}^{\top} = (\mathbf{d}_{n,1,r}^{\top},..., \mathbf{d}_{n,j,l}^{\top}, \mathbf{d}_{n,j,r}^{\top}, ..., \mathbf{d}_{n,N,r}^{\top})^{\top}$ with
    \begin{equation}
      \label{dRobin}
      \begin{split}
        \mathbf{d}_{n,j,l} & = 2ip \cdot {Q}_{j-1,r} (\mathbb{A}_{j-1,n} +ip \cdot \mathbb{M}^{\Gamma_{j-1}}) \frac{2i}{\Delta t} \mathbf{u}_{j-1,n}, \ j=2,3,...,N, \\
        \mathbf{d}_{n,j,r} & = 2ip \cdot {Q}_{j+1,l} (\mathbb{A}_{j+1,n} +ip \cdot \mathbb{M}^{\Gamma_{j+1}})^{-1} \frac{2i}{\Delta t} \mathbf{u}_{j+1,n}, \ j=1,2,...,N-1.
      \end{split}
    \end{equation}
    Secondly, since $V=0$, then 
    \begin{equation}
      \label{Aeq}
      \mathbb{M}_{j,W_n}=0, \quad \mathbb{A}_{j,1}=\mathbb{A}_{j,2}=...=\mathbb{A}_{j,N_T} = \frac{2i}{\Delta t} \mathbb{M}_{j}-\mathbb{S}_{j}, j=1,2,...,N.
    \end{equation}
    Thus, the blocks $X^{j,1}$, $X^{j,2}$, $X^{j,3}$ and $X^{j,4}$ are both independent of time step $n$.

    Finally, thanks to the hypothesis of the proposition, the geometry of each subdomain is identical. Thus, the various matrices coming from the assembly of the finite element methods are the same. Therefore, we have
    \begin{gather}
      \mathbb{M}_1=\mathbb{M}_2=...=\mathbb{M}_N, \quad \nonumber
      \mathbb{S}_1=\mathbb{S}_2=...=\mathbb{S}_N, \quad
      \mathbb{M}^{\Gamma_{1}}=\mathbb{M}^{\Gamma_{2}}=...=\mathbb{M}^{\Gamma_{N}}, \\
      Q_{1,l}=Q_{2,l}=...=Q_{N,l}, \quad
      Q_{1,r}=Q_{2,r}=...=Q_{N,r}. \label{Meq}
    \end{gather} 
    The conclusion follows directly from \eqref{matXRobin}.
  \end{proof}

  \begin{proposition}
    \label{props2p}
    Let us consider the transmission condition $S_{\mathrm{pade}}^m$. If the potential is zero, then the matrix $\mathcal{L}_h$ takes the form \eqref{Lmpi} and $\mathcal{L}_h$ is independent of time step $n$. In addition, if the subdomains are equal, then \eqref{blockeq} is true.
  \end{proposition}                            
  \begin{proof}
    The proof is almost same as that of the proposition \ref{proprobin}. Using \eqref{dists2p} and \eqref{AlgoGraph} gives
    \begin{equation}
      \label{matXSP}
      \begin{aligned}
        & X^{j,1} = -I - 2{Q}_{j,l} \Big( i\sum_{s=0}^m a_s^m + i\sum_{s=1}^m a_s^m d_s^m \mathbb{D}_{j,n,s}^{-1} \mathbb{C}_j \Big) \Big( \mathbb{A}_{j,n} + i(\sum_{s=1}^m a_s^m)  \cdot \mathbb{M}^{\Gamma_j}
        - \sum_{s=1} \mathbb{B}_{j,s} \mathbb{D}_{j,n,s}^{-1} \mathbb{C}_{j} \Big)^{-1} \mathbb{M}^{\Gamma_{j}} {Q}_{j,l}^{\top}, \\
        & X^{j,2} = - 2{Q}_{j,l} \Big( i\sum_{s=0}^m a_s^m + i\sum_{s=1}^m a_s^m d_s^m \mathbb{D}_{j,n,s}^{-1} \mathbb{C}_j \Big) \Big( \mathbb{A}_{j,n} + i(\sum_{s=1}^m a_s^m)  \cdot \mathbb{M}^{\Gamma_j}
        - \sum_{s=1} \mathbb{B}_{j,s} \mathbb{D}_{j,n,s}^{-1} \mathbb{C}_{j} \Big)^{-1} \mathbb{M}^{\Gamma_{j}} {Q}_{j,r}^{\top}, \\
        & X^{j,3} = - 2{Q}_{j,r} \Big( i\sum_{s=0}^m a_s^m + i\sum_{s=1}^m a_s^m d_s^m \mathbb{D}_{j,n,s}^{-1} \mathbb{C}_j \Big) \Big( \mathbb{A}_{j,n} + i(\sum_{s=1}^m a_s^m)  \cdot \mathbb{M}^{\Gamma_j}
        - \sum_{s=1} \mathbb{B}_{j,s} \mathbb{D}_{j,n,s}^{-1} \mathbb{C}_{j} \Big)^{-1} \mathbb{M}^{\Gamma_{j}} {Q}_{j,l}^{\top}, \\
        & X^{j,4} = -I - 2{Q}_{j,r} \Big( i\sum_{s=0}^m a_s^m + i\sum_{s=1}^m a_s^m d_s^m \mathbb{D}_{j,n,s}^{-1} \mathbb{C}_j \Big) \Big( \mathbb{A}_{j,n} + i(\sum_{s=1}^m a_s^m)  \cdot \mathbb{M}^{\Gamma_j}
        - \sum_{s=1} \mathbb{B}_{j,s} \mathbb{D}_{j,n,s}^{-1} \mathbb{C}_{j} \Big)^{-1} \mathbb{M}^{\Gamma_{j}} {Q}_{j,r}^{\top}.
      \end{aligned}
    \end{equation}
    Under these assumptions, we have \eqref{Aeq} and \eqref{Meq}. In addition, the matrix associated with the transmission condition $S_{\mathrm{pade}}^m$ satisfy
    \begin{displaymath}
      \mathbb{B}_{1,s}=\mathbb{B}_{2,s}=...=\mathbb{B}_{N,s}, \quad
      \mathbb{C}_{1}=\mathbb{C}_{2}=...=\mathbb{C}_{N}, \quad
      \mathbb{D}_{1,n,s}=\mathbb{D}_{2,n,s}=...=\mathbb{D}_{N,n,s}, \ s=1,2,...m,
    \end{displaymath}
    and $\mathbb{D}_{j,n,s} = Q (\frac{2i}{\Delta t} \mathbb{M}^{\Gamma_j} -  \mathbb{S}^{\Gamma_j} + \mathbb{M}^{\Gamma_j}_{W_n} + d_s^m \mathbb{M}^{\Gamma_j}) Q^{\top}$ is independent of time step $n$ since  $\mathbb{M}^{\Gamma_j}_{W_n}=0$.
  \end{proof}

  Intuitively, the free semi-discrete Schr{\"o}dinger operator without
  potential is a rough approximation of the semi-discrete
  Schr{\"o}dinger operator with potential: 
  \begin{displaymath}
    \frac{2i}{\Delta t} u + \Delta u \approx \frac{2i}{\Delta t} u + \Delta u + V u + f(u)u.
  \end{displaymath} 
  In other words, $V u + f(u)u$ is a perturbation of the free
  semi-discrete Schr\"odinger operator. Thus, the matrix
  $\mathcal{L}_{h}$ can be seen as an approximation to
  \begin{displaymath}
    I-\mathcal{R}_{h,n}.
  \end{displaymath}
  Based on the previous propositions, it is sufficient to compute only four
  blocks $X^{2,1}$, $X^{2,2}$, $X^{2,3}$ and $X^{2,4}$ to construct the
  preconditioner $P$. It will be shown in the following section that the
  construction can be implemented in a parallel way.

  \section{Parallel implementation}
  \label{Sec_Imp}

  We present the parallel implementation of the classical
  algorithm \eqref{algo_d} and the preconditioned algorithm
  \eqref{algo_pd} in this section. We fix one MPI process per
  subdomain \cite{mpiforum30}. We use the distributed matrix,
  vector and iterative linear system solver avalaible in PETSc library
  \cite{petsc-user-ref}. 

  \subsection{Classical algorithm}

  The discrete interface vector $\mathbf{g}_{n}^k$ is stored in
  a distributed manner in PETSc form. As shown by
  \eqref{gn_petsc}, $\mathbf{r}_{1,n}^k$ is located in MPI process 0, $\mathbf{l}_{j,n}^k$ and $\mathbf{r}_{j,n}^k$ are in MPI process $j-1$, $j=2,3,...,N-1$ and $\mathbf{r}_{N,n}^k$ is in MPI process $N-1$.
  \begin{equation} 
    \label{gn_petsc}
    \mathbf{g}_{n}^k = 
    \begin{pmatrix}
      \mathbf{r}_{1,n}^k \\
      \vdots \\
      \mathbf{l}_{j,n}^k \\
      \mathbf{r}_{j,n}^k \\
      \vdots \\
      \mathbf{r}_{N,n}^k
    \end{pmatrix}
    \renewcommand{\arraystretch}{1.15}
    \begin{array}{c} 
      \rdelim \}{1}{4pt}[\small MPI 0] \\  
      \\
      \rdelim \}{2}{4pt}[\small MPI $j-1$] \\
      \\
      \\
      \rdelim \}{1}{4pt}[\small MPI $N-1$] \\
    \end{array}
  \end{equation} 
  As shown by \eqref{AlgoGraph} for $N=3$, at iteration $k$, $\mathbf{v}_{j,n}^k,j=1,2,...,N$ is computed on each subdomain locally and the boundary values are communicated.

  \subsection{Preconditioned algorithm}
  Thanks to the analysis yielded in previous section, we can build
  explicitly $\mathcal{L}_h$ with few computations. For the Robin
  transmission condition, it is is based on the formulas
  \eqref{matXRobin}. For the transmission condition
  $S_{\mathrm{pade}}^m$, the idea is equivalent, but involves
  \eqref{matXSP}. 
  According to the proposition \eqref{proprobin}, the column $s$ of $X^{2,1}$ and $X^{2,3}$ are 
  \begin{displaymath}
    \begin{split}
      X^{2,1} \mathbf{e}_s & = -\mathbf{e}_s - 2ip \cdot Q_{2,l} (\mathbb{A}_{j,n} + ip \cdot \mathbb{M}^{\Gamma_j})^{-1} \mathbb{M}^{\Gamma_{2}} Q_{2,l}^{\top} \mathbf{e}_s, \\
      X^{2,3} \mathbf{e}_s & = - 2ip \cdot Q_{2,r}  (\mathbb{A}_{j,n} + ip \cdot \mathbb{M}^{\Gamma_j})^{-1} \mathbb{M}^{\Gamma_{2}} Q_{2,l}^{\top} \mathbf{e}_s,
    \end{split}
  \end{displaymath}                               
  where
  $\mathbf{e}_s =(0,0,...,1,...0) \in
  \mathbb{C}^{N_T \times N_y}$, all its elements
  are zero except the $s$-th, which is one. The
  element $\mathbb{M}^{\Gamma_{2}}
  Q_{2,l}^{\top} \mathbf{e}_s$ being a vector,
  it is necessary to compute one time the
  application of $ (\mathbb{A}_{j,n} + ip \cdot
  \mathbb{M}^{\Gamma_j})^{-1}$ to
  vector. Similarly, we have 
  \begin{displaymath}
    \begin{split}
      X^{2,2} \mathbf{e}_s & = - 2ip \cdot Q_{2,l}  (\mathbb{A}_{j,n} + ip \cdot \mathbb{M}^{\Gamma_j})^{-1} \mathbb{M}^{\Gamma_{2}} Q_{2,r}^{\top} \mathbf{e}_s, \\
      X^{2,4} \mathbf{e}_s & = -\mathbf{e}_s - 2ip \cdot Q_{2,r}  (\mathbb{A}_{j,n} + ip \cdot \mathbb{M}^{\Gamma_j})^{-1}
      \mathbb{M}^{\Gamma_{2}} Q_{2,r}^{\top} \mathbf{e}_s,
    \end{split}
  \end{displaymath} 
  Let us recall that $\mathbb{A}_{j,n}=\frac{2i}{\Delta
    t}\mathbb{M}_{j}-\mathbb{S}_{j}$ for $V=0$, $f=0$. To know
  the first $N_y$ columns of $X^{2,1}$, $X^{2,2}$, $X^{2,3}$ and
  $X^{2,4}$, we only have to compute $2N_y$ times the application of $
  (\mathbb{A}_{j,n} + ip \cdot \mathbb{M}^{\Gamma_j})^{-1}$ to
  vector. In other words, this amounts to solve the Schr\"odinger
  equation on a single subdomain $2N_y$ times to build the matrix
  $\mathcal{L}_h$. The resolutions are all independent. Therefore, we
  can solve them on different processors using 
  MPI paradigm. We fix one MPI process per domain. To construct the
  matrix $\mathcal{L}_h$, we use the $N$ MPI processes to solve the
  equation on a single subdomain (ex. $(0,T) \times \Omega_2$) $2N_y$
  times. Each MPI process therefore solves the Schr\"{o}dinger equation
  on a single subdomain maximum 
  \begin{displaymath}
    N_{\mathrm{mpi}} := [\frac{2N_y}{N}] + 1 \ \text{times,}
  \end{displaymath}
  where $[x]$ is the integer part of $x$. This construction is therefore super-scalable in theory. Indeed, if $N$ is doubled, then the size of subdomain is divided by two and $N_{\mathrm {mpi}} $ is also approximately halved.                   

  Concerning the computational phase, the transpose of $\mathcal{L}_{h}$ is stored in a distributed manner using the library PETSc. As shown by \eqref{Lmpi}, the first block column of $\mathcal{L}_{h}$ lies in MPI process 0. The second and third blocks columns are in MPI process 1, and so on for other processes. In addition, for any vector $y$, the vector $x: = P^{-1} y$ is computed by solving the linear system
  \begin{displaymath}
    P x = y
  \end{displaymath}  
  with Krylov methods (GMRES or BiCGStab).



  \section{Numerical results}
  \label{Sec_Num}

  We implement the algorithms in a cluster consisting of 92 nodes (16 cores/node, Intel Sandy Bridge E5-2670, 32GB/node). We fix one MPI process per subdomain and 16 MPI processes per node. The communications are handled by PETSc and Intel MPI. The linear systems related to \eqref{NproblemRobinNL}, \eqref{NproblemSPNL}, \eqref{NproblemRobinL} and \eqref{NproblemSPL},  are solved with the LU direct method using MKL Pardiso library. The convergence condition for our algorithms is $\parallel \mathbf{g}_{n}^{k+1} - \mathbf{g}_{n}^k \parallel <10^{-10},n=1,2,...,N_T$. The initial vectors are
  \begin{itemize}
  \item $\mathbf{g}_{1}^0=\mathbf{0}$ or $\mathbf{g}_{1}^0=$ random vector,
  \item $\displaystyle \mathbf{g}_{n}^0= \lim_{k\rightarrow \infty} \mathbf{g}_{n-1}^{k}, n=2,3,...,N_T$.
  \end{itemize}
  Since the convergence properties for different time steps $n=1,2,...,N_T$ are similar, we only consider the number of iterations required for convergence of the first time step $n=1$.  As mentioned in \cite{Gander2008history}, using the zero initial vector could give wrong conclusions associated with the convergence. Thus, the zero vector is used when one wants to evaluate the computation time, while the random vector is used when comparing the transmission conditions. The theoretical optimal parameter $p$ (resp. $m$) in the transmission condition Robin (resp. $S_{\mathrm{pade}}^m$) is not at hand for us, we then seek the best parameter numerically for each case.

  This section is composed of two subsections. The first one is devoted to the Schr\"odinger equation. In the second, we consider the simulation of Bose-Einstein condensates.

  \subsection{Schr\"{o}dinger equation}
  \label{Sec_Num_Sch}

  We decompose the physical domain $(-16,16) \times (-8,8)$ into $N$ equal subdomains without overlap. The final time $T$ and the time step $\Delta t$ are fixed to be $T=0.5$ and $\Delta t = 0.01$ in this subsection. We consider two different meshes
  \begin{gather*}
    \Delta x=1/128, \ \Delta y=1/8,\\
    \Delta x=1/2048, \ \Delta y=1/64,
  \end{gather*}
  where the size of cell is $\Delta x \times \Delta y$. The potential and the initial datum (Figure \ref{initsol}) are
  \begin{displaymath}
    \mathscr{V}=|u|^2, \quad u_0(x,y) = e^{-x^2-y^2 -0.5i x},
  \end{displaymath} 
  which give rise to a solution that propagates slowly to the negative side in $y$ direction and undergoes dispersion. It is possible to solve numerically the Schr\"{o}dinger equation on the entire domain $\Omega$ with the first mesh ($\Delta x=1/128, \Delta y=1/8$) under the memory limitation (32G). We compare in this sub section the classical and the preconditioned algorithms, as well as the two transmission conditions.
  \begin{figure}[!htbp]
    \centering
    \includegraphics[width=0.5\textwidth]{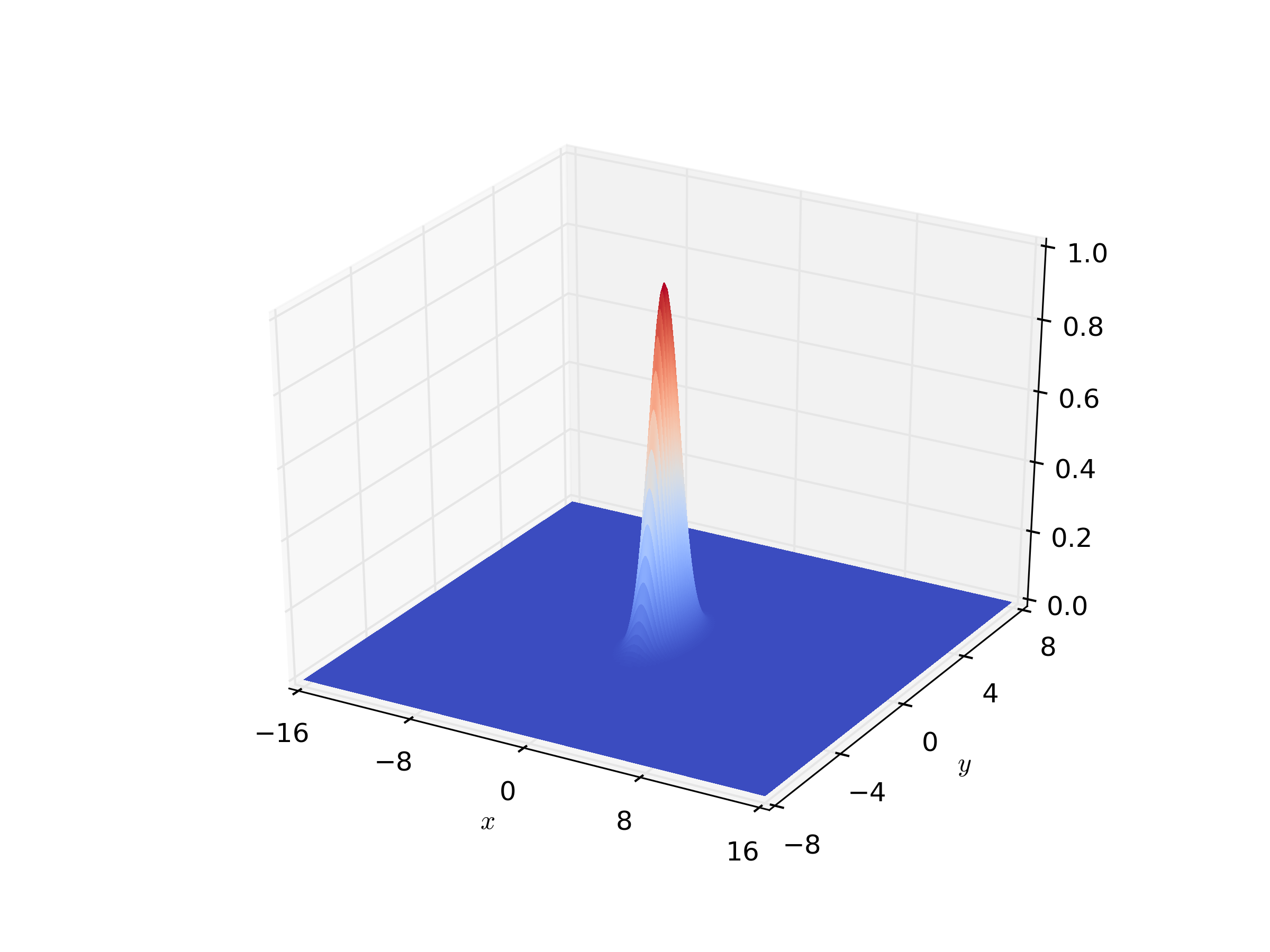}
    \caption{Initial datum $|u_0|$.}
    \label{initsol}
  \end{figure}

  \subsubsection{Comparison of the classical algorithm and the preconditioned algorithm}

  We are interested in observing the robustness, the number of iterations of the first time step, the computation time involving the transmission condition $S_{\mathrm{pade}}^m$. The zero vector is used as the initial vector $\mathbf{g}_{1}^0$. We denote by $N_{\mathrm{nopc}}$ (resp. $N_{\mathrm{pc}}$) the number of iterations required for convergence with the classical algorithm (resp. the preconditioned algorithm). $T_{\mathrm{nopc}}$ and  $T_{\mathrm{pc}}$ denote the computation times of the classical algorithm and the preconditioned algorithm respectively. In addition, we denote by $T^{\mathrm{ref}}$ the computation time to solve numerically on a single processor the Schr{\"o}dinger equation on the entire domain. 

  First, we consider a mesh with $\Delta x=1/128$, $\Delta
  y=1/8$. We make the tests for $N=2,4,8,16,32$ subdomains. The
  convergence history of the first time step is presented in
  Figure \ref{hist_2_32_NL_abc2p} for $N=2$ (left) and $N=32$
  (right). Table \ref{Niter_Time_2_32_NL_abc2p} shows the number
  of iterations of the first time step and the computation
  times. We can see that all the algorithms are robust and
  scalable. The number of iterations is independent of number of
  subdomains. This independence has already been observed for one
  dimensional Schr{\"o}dinger equation for small $N$
  \cite{Halpern2010_sch, XF20151d}. In addition, the
  preconditioner allow to reduce number of iterations and
  computation time. 

  \begin{figure}[!htbp]
    \centering
    \includegraphics[width=0.4\textwidth]{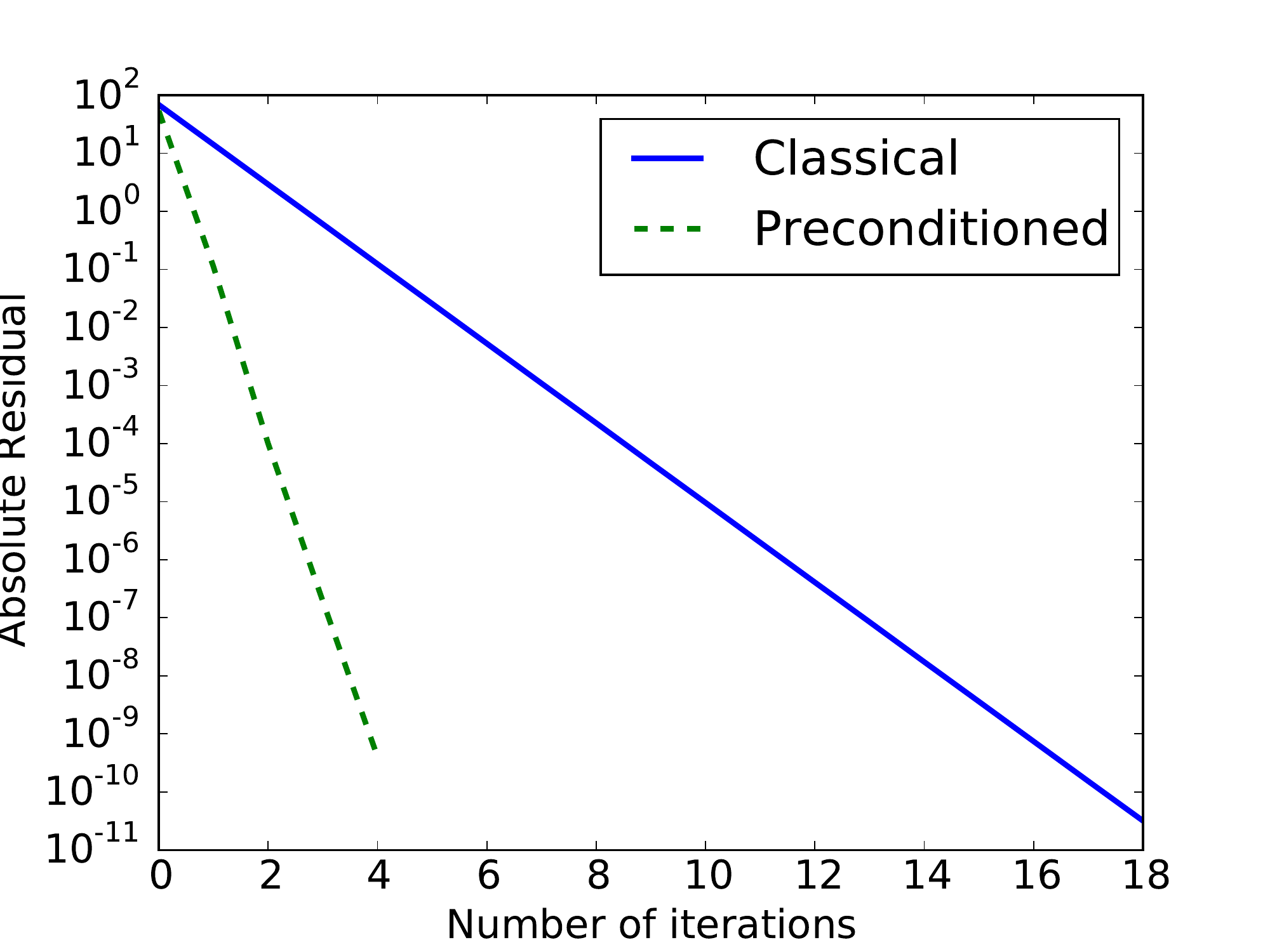}
    \includegraphics[width=0.4\textwidth]{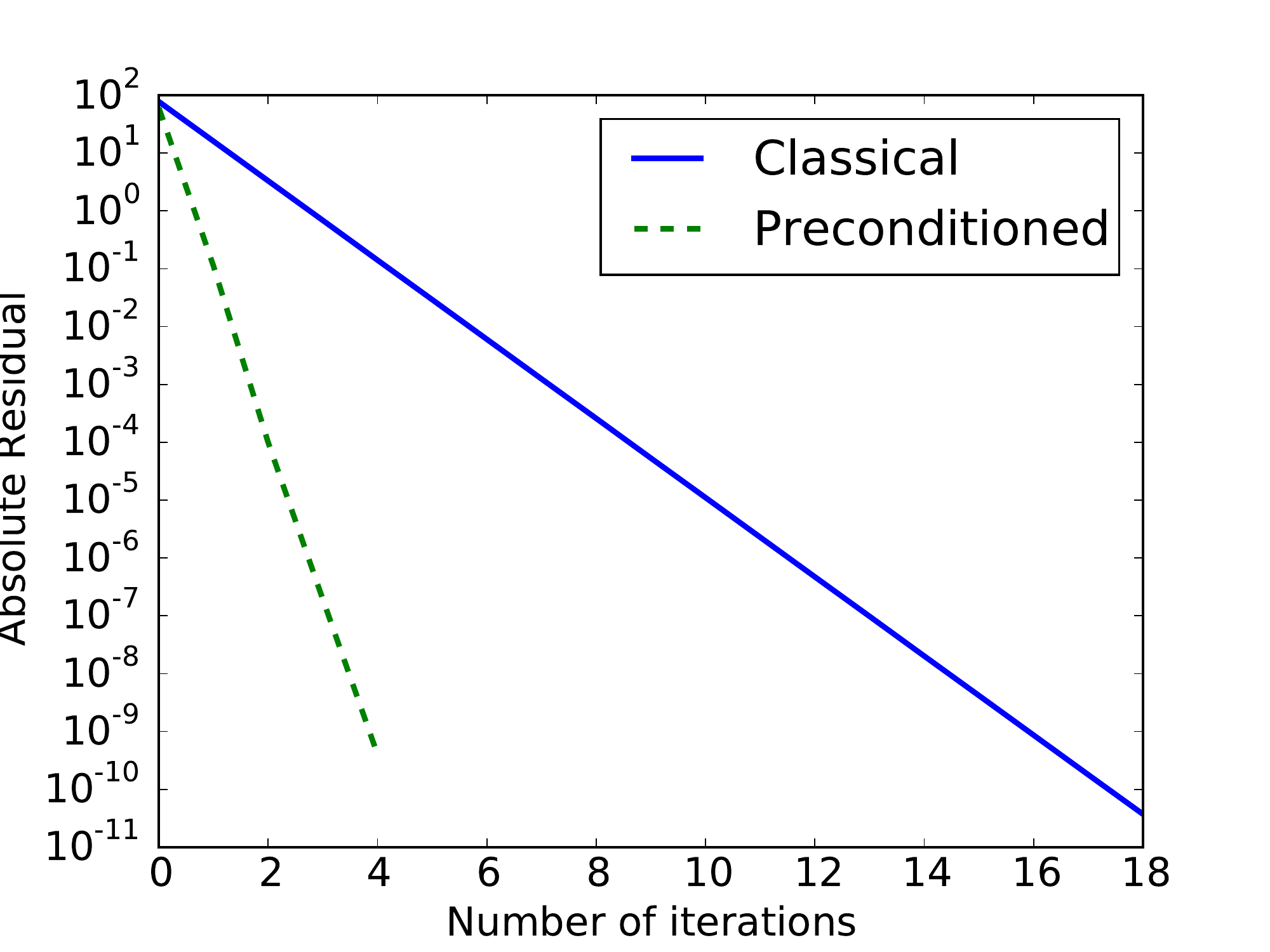}
    \caption{\small Convergence histories of the first time step for $N=2$ (left) and $N=32$ (right). The mesh is $\Delta x = 1/128$, $\Delta y = 1/8$.}
    \label{hist_2_32_NL_abc2p}
  \end{figure}

  \begin{table}[!htbp] 
    \renewcommand{\arraystretch}{1.2}
    \centering
    \begin{tabular}{|c|c|c|c|c|c|}
      \hline
      $N$ & 2 & 4 & 8 & 16 & 32 \\
      \hline
      $N_{\mathrm{nopc}}$ & 18 & 18 & 18 & 18 & 18 \\ 
      \hline
      $N_{\mathrm{pc}}$ & 6 & 6 & 6 & 6 & 6 \\
      \hline
      $T^{\mathrm{ref}}$ & \multicolumn{5}{c|}{93.7}\\
      \hline
      $T_{\mathrm{nopc}}$ & 1106.2 & 571.2 & 297.4 & 161.8 & 85.2 \\ 
      \hline
      $T_{\mathrm{pc}}$ & 356.2 & 180.5 & 92.1 & 50.3 & 26.6 \\ 
      \hline
    \end{tabular}
    \caption{\small Number of iterations and total computation
      time (seconds) of the algorithms with the mesh $\Delta x =
      1/128$, $\Delta y = 1/8$.}
    \label{Niter_Time_2_32_NL_abc2p}
  \end{table}

  Secondly, we reproduce the same tests with the mesh $\Delta x
  = 1/2048$, $\Delta y = 1/64$. The convergence history of the
  first time step, the total computation times are shown in
  Figure \ref{hist_N_256_1024_NL_abc2p} and Table
  \ref{Niter_Time_256_1024_NL_abc2p}. The algorithms are both
  robust for $N=1024$, but not scalable from $N=512$ to
  $N=1024$. The classical algorithm loses scalability since if
  we use more
  subdomains used to decompose $\Omega$, then more iterations are
  required for convergence. {Concerning the preconditioned
    algorithm, the computational time is larger with $N=1024$
    compared to $N=512$ since the application of the preconditioner increases with larger $N$.} However, the preconditioned algorithm is
  much more efficient since it can both
  reduce  the 
  number of iterations and the total computation times. 

  \begin{figure}[!htbp]
    \centering
    \includegraphics[width=0.4\textwidth]{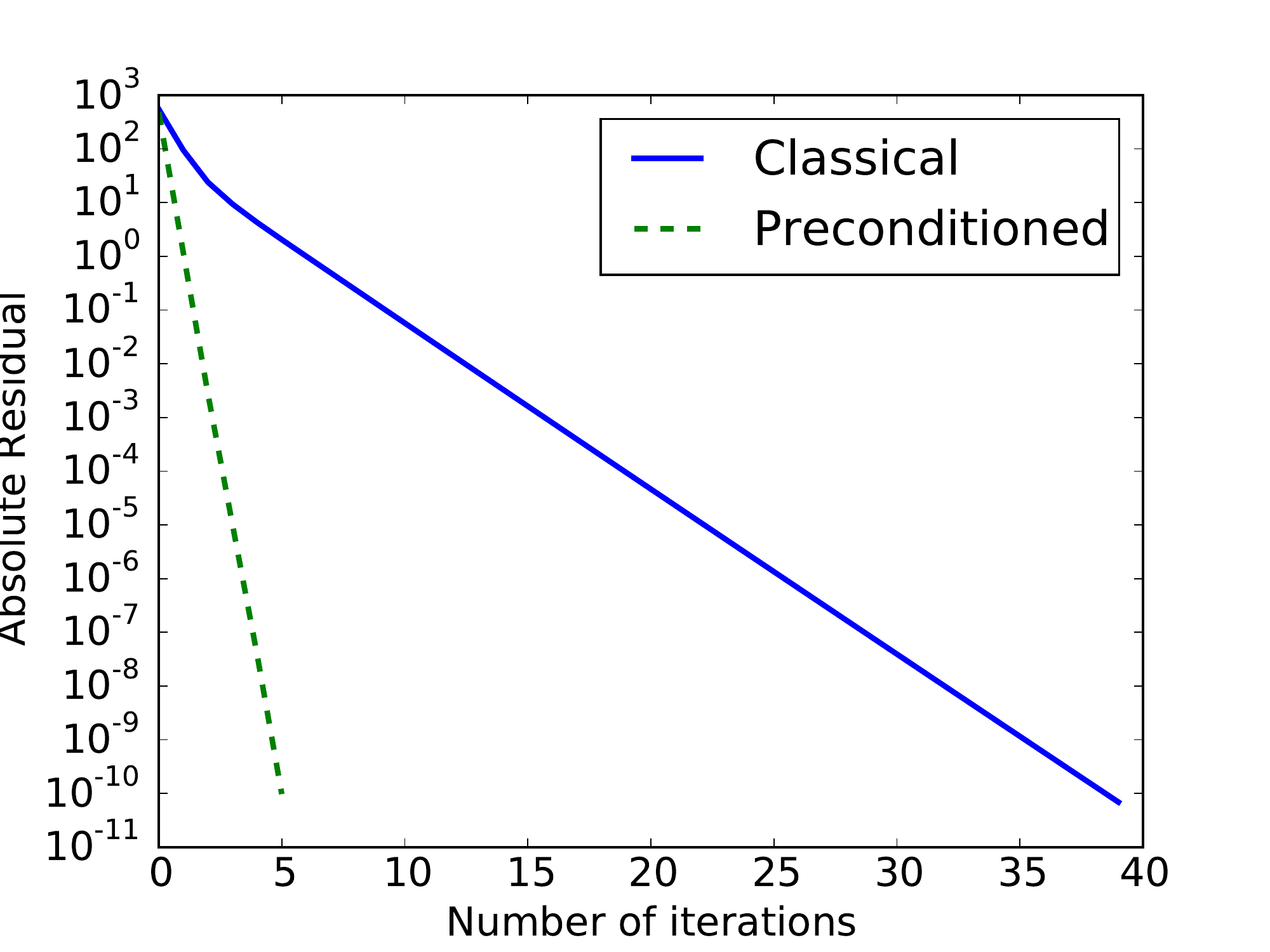}
    \includegraphics[width=0.4\textwidth]{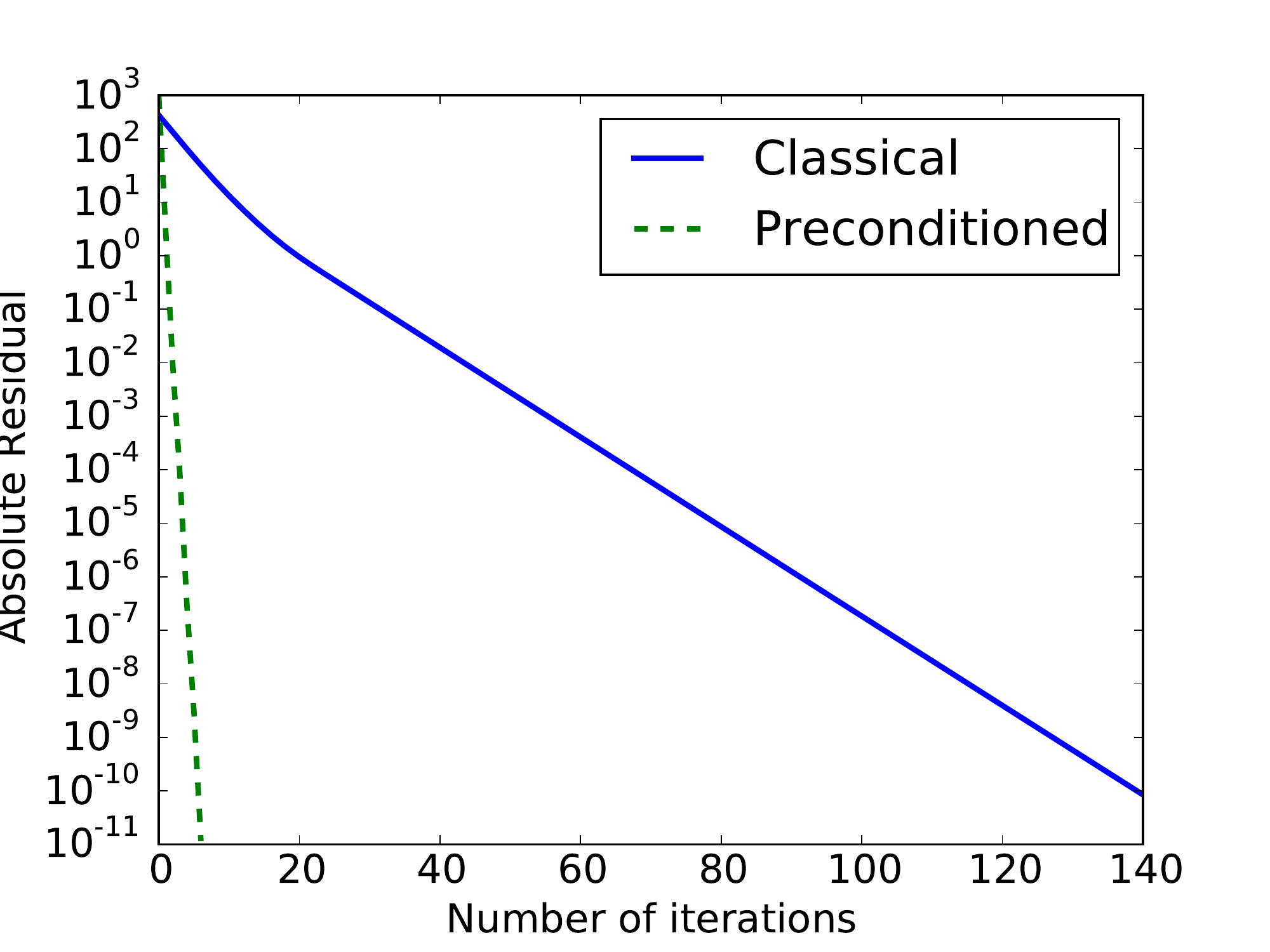}
    \caption{\small Convergence histories of the first time step for $N=256$ (left) and $N=1024$ (right) with the mesh $\Delta x = 1/2048$, $\Delta y = 1/64$.}
    \label{hist_N_256_1024_NL_abc2p}
  \end{figure}

  \begin{table}[!htbp] 
    \renewcommand{\arraystretch}{1.2}
    \centering
    \begin{tabular}{|c|c|c|c|c|}
      \hline
      $N$ & 256 & 512 & 1024 \\
      \hline
      Classical algorithm & 3582.4 & 2681.5 & 2516.6 \\
      \hline
      Preconditioned algorithm & 596.9 & 376.1 & 441.9\\
      \hline
    \end{tabular}
    \caption{\small Computation times (seconds) of the algorithms with the mesh $ \Delta x = 1/2048, \Delta y = 1/64$.}
    \label{Niter_Time_256_1024_NL_abc2p}
  \end{table}


  \subsubsection{Comparison of transmission conditions}
  
  In this part, we compare numerically the transmission
  conditions Robin and $S_{\mathrm{pade}}^m$ in the framework of
  the two algorithms. The initial vector $\mathbf{g}_{1}^0$ here
  is a random vector to make sure that all the frequencies are
  included. The time step is fixed to be $\Delta t=0.01$ and the
  mesh is $\Delta x = 1/128$, $\Delta y = 1/8$. Figure
  \ref{compa_hist_NL_N2} and Figure \ref{compa_hist_NL_N32}
  present the convergence histories of the first time step in
  the framework of the classical and the preconditioned
  algorithms with Robin and $S_{\mathrm{pade}}^m$ transmission
  conditions for $N=2$ and $N=32$ respectively. It can be seen
  that in the framework of the classical algorithm, the
  transmission condition $S_{\mathrm{pade}}^m$ allows the
  algorithm to  converge faster, while in the framework of
  preconditioned algorithm, they have similar histories of
  convergence. This observation indicates that the
  preconditioner $P$ is a good approximation of the nonlinear
  operator $I - \mathcal{R}_{h,n}$. The influence of the
  transmission conditions is eliminated by the
  preconditioner. In addition, we could confirm the conclusion of
  the previous subsection: the preconditioner reduces a lot the
  number of iterations required for convergence.  
  \begin{figure}[!htbp]
    \centering
    \includegraphics[width=0.4\textwidth]{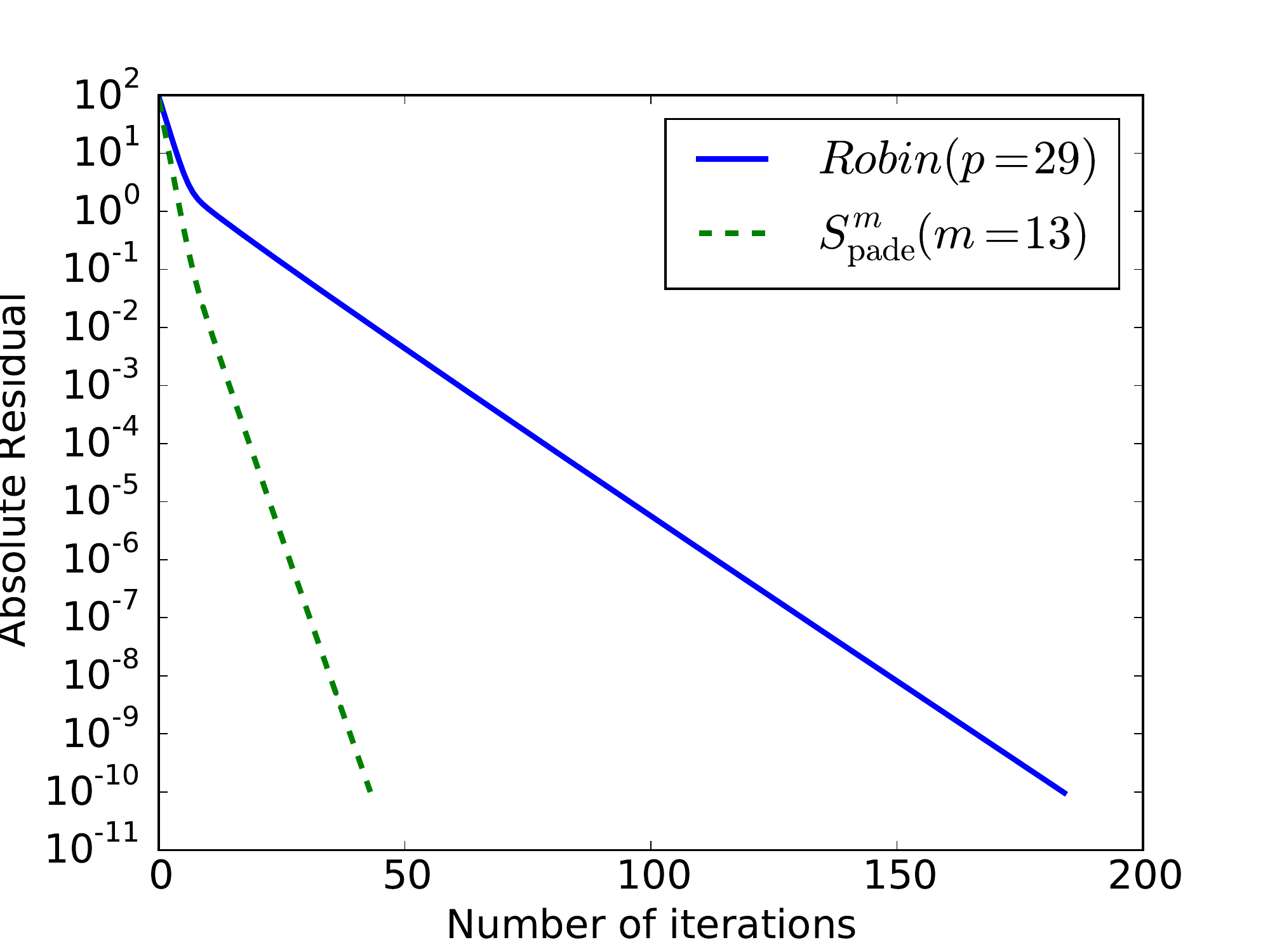}
    \includegraphics[width=0.4\textwidth]{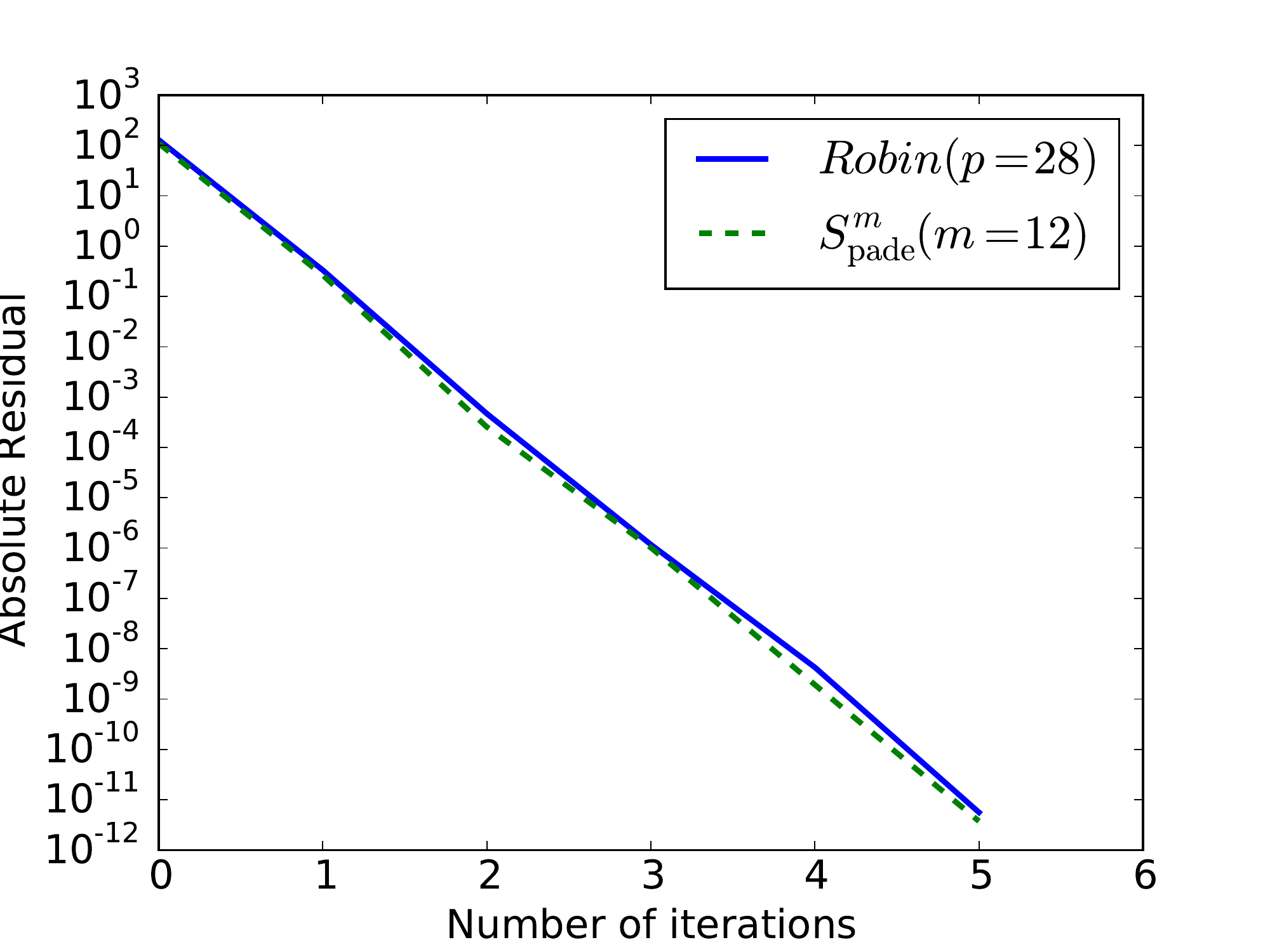}
    \caption{\small Convergence histories of the first time step of the classical algorithm (left) and the preconditioned algorithm for $N=2$. The mesh is $\Delta x = 1/128$, $\Delta y = 1/8$.}
    \label{compa_hist_NL_N2}
  \end{figure}
  \begin{figure}[!htbp]
    \centering
    \includegraphics[width=0.4\textwidth]{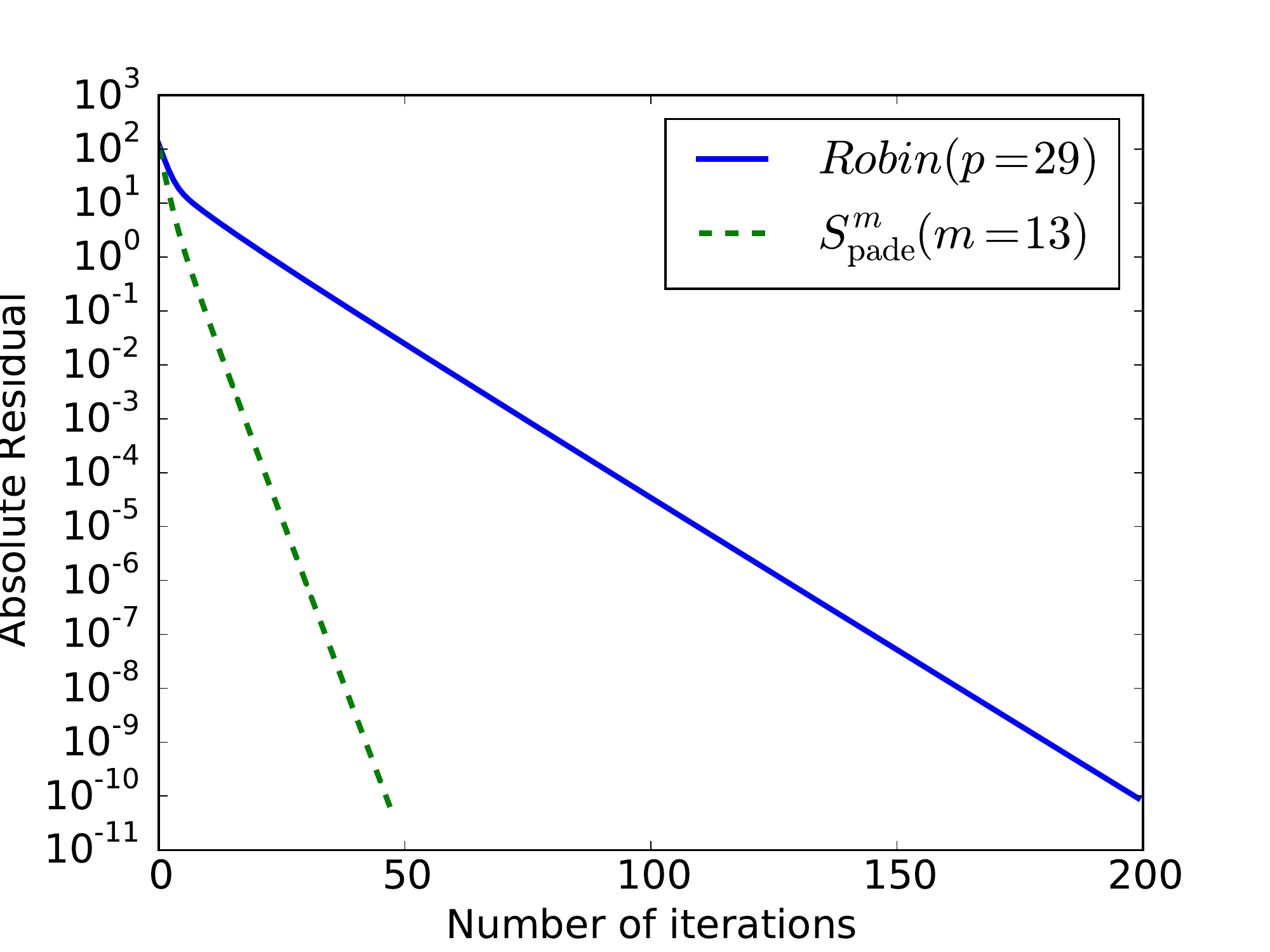}
    \includegraphics[width=0.4\textwidth]{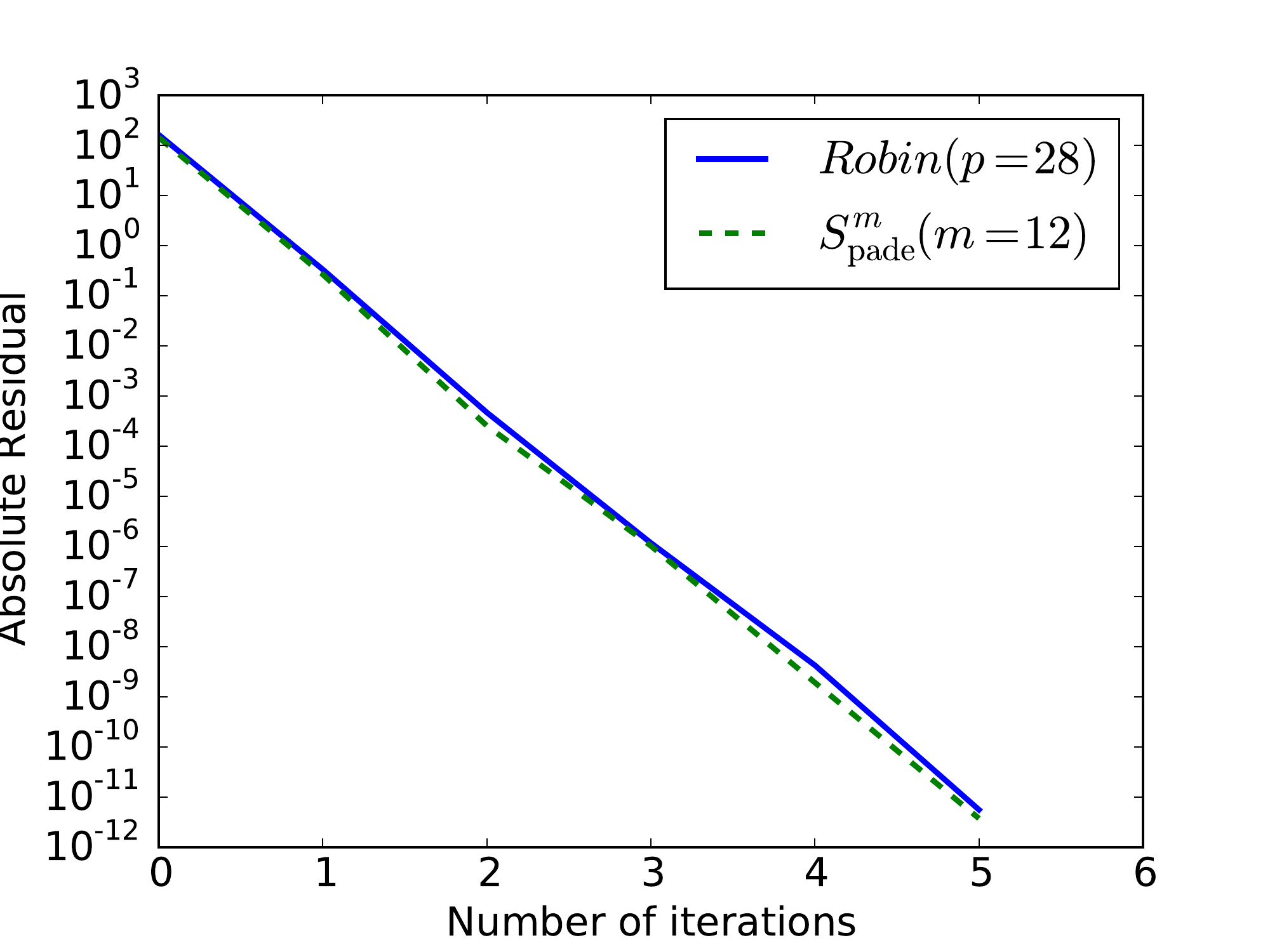}
    \caption{\small Convergence histories of the first time step of the classical algorithm (left) and the preconditioned algorithm for $N=32$. The mesh is $\Delta x = 1/128$, $\Delta y = 1/8$.}
    \label{compa_hist_NL_N32}
  \end{figure}%

  \subsubsection{Influence of parameters}

  In this subsection, we study the influence of parameters in the transmission conditions:
  \begin{itemize}
  \item the parameter $m$ (order of Pad\'{e} approximation) in the transmission condition $S_{\mathrm{pade}}^m$,
  \item the parameter $p$ in the transmission condition Robin.
  \end{itemize}
  The time step and the mesh are fixed to be $\Delta t=0.01$ and $\Delta x=1/128$, $\Delta y=1/8$.

  Firstly, we consider the influence of $m$ in the transmission condition $S_{\mathrm{pade}}^m$. We present in Figure \ref{Influence_m_NL_cls} and in Figure \ref{Influence_m_NL_pd} the number of iterations in relation to the order of Pad\'{e} approximation ($m$) in the framework of the classical and the preconditioned algorithms. Both of the zero vector and the random vector are considered as the initial vector in our tests. 
  \begin{itemize}
  \item
    For the classical algorithm, if the initial vector is the
    zero vector, there exists an optimal parameter $m$. This
    observation is not consistent with our expectations since
    the higher order should make the algorithm converge
    faster. We believe that the zero initial vector gives
    us some inaccurate information. 
  \item
    For the classical algorithm, if the initial vector is a
    random vector, the number of iterations first decreases then
    increases by increasing the order $m$. {We however do not have yet an explanation for the relation between the convergence and the parameter $m$, which needs some more investigations.}
    
  \item
    The parameter $m$ is not very important for the preconditioned algorithm since the preconditioner hides the information about the order.
  \end{itemize}
  \begin{figure}[!htbp]
    \centering
    \includegraphics[width=0.4\textwidth]{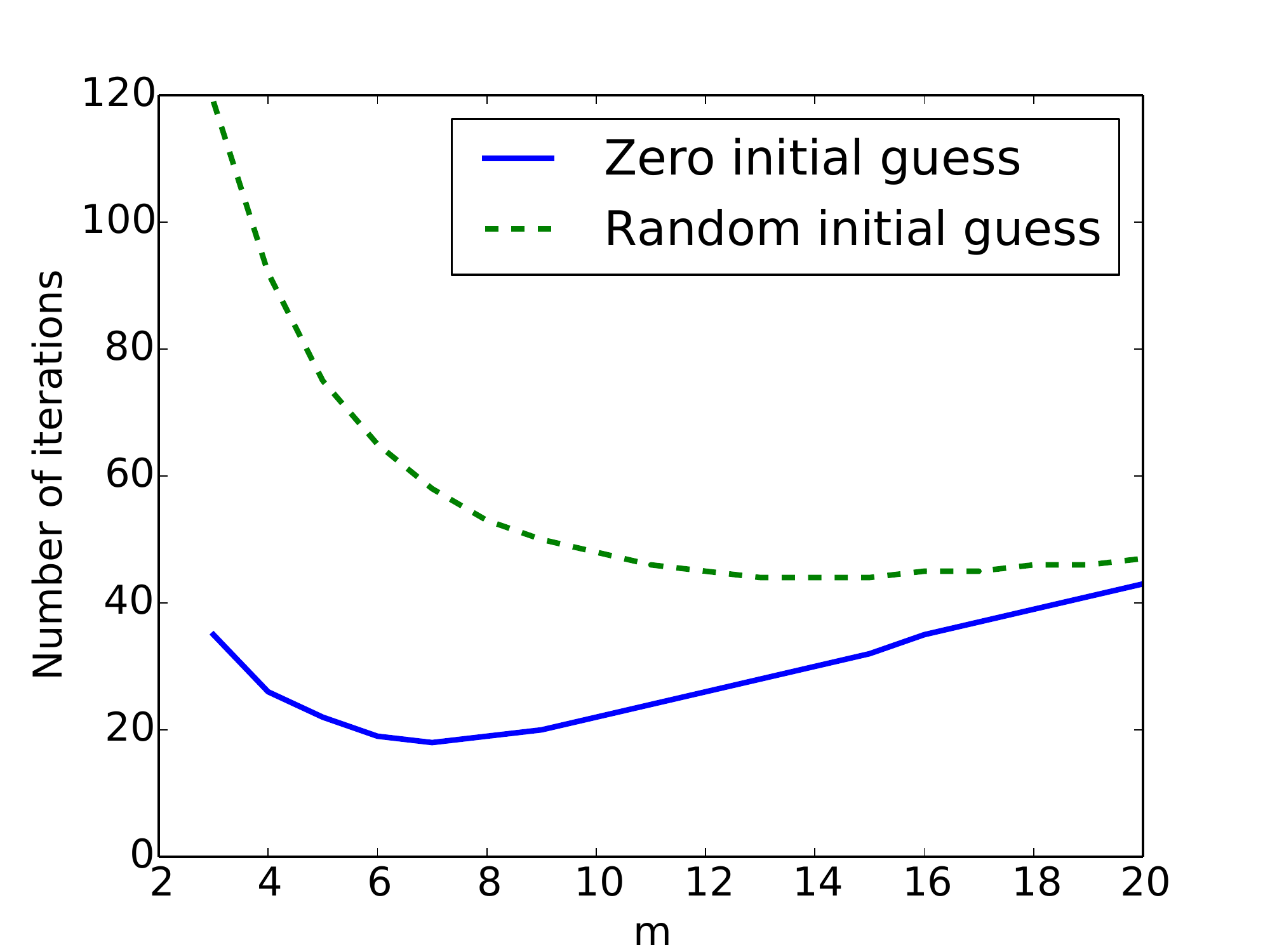}
    \includegraphics[width=0.4\textwidth]{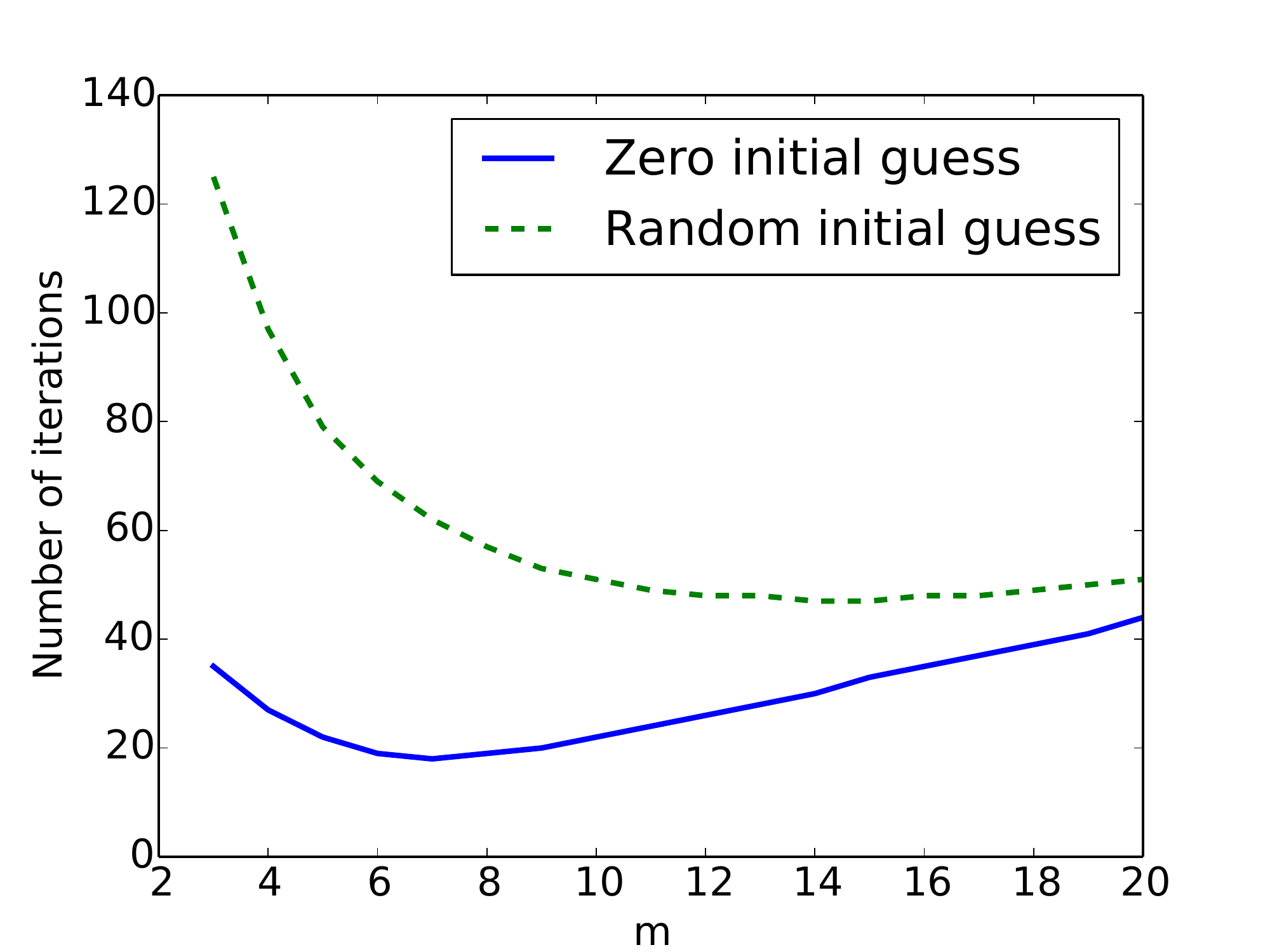}
    \caption{\small Number of iterations vs. parameter $m$ for $N=2$
      (left) and $N=32$ (right) in the framework of the classical
      algorithm.}
    \label{Influence_m_NL_cls}
  \end{figure}%
  \begin{figure}[!htbp]
    \centering
    \includegraphics[width=0.4\textwidth]{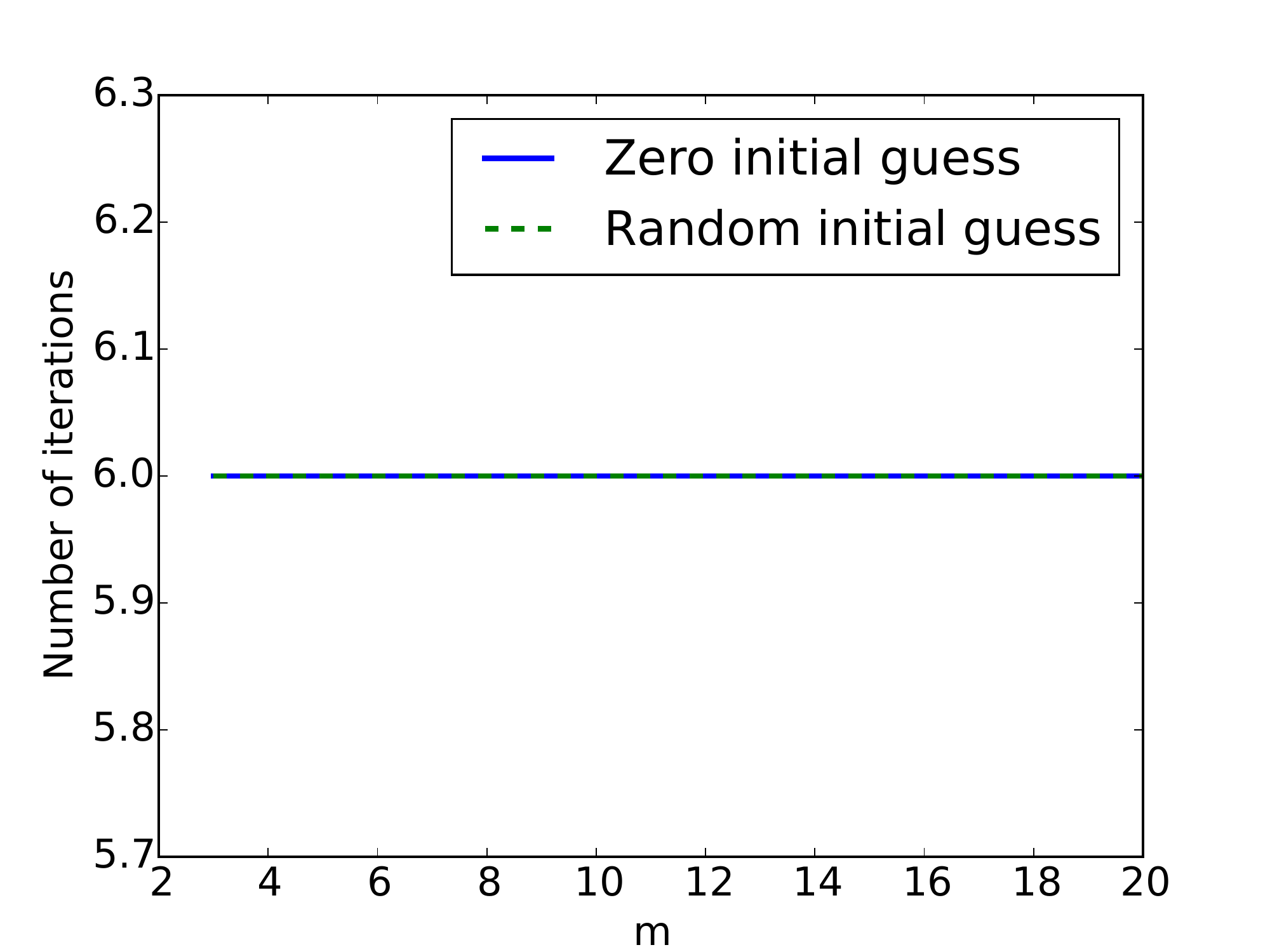}
    \includegraphics[width=0.4\textwidth]{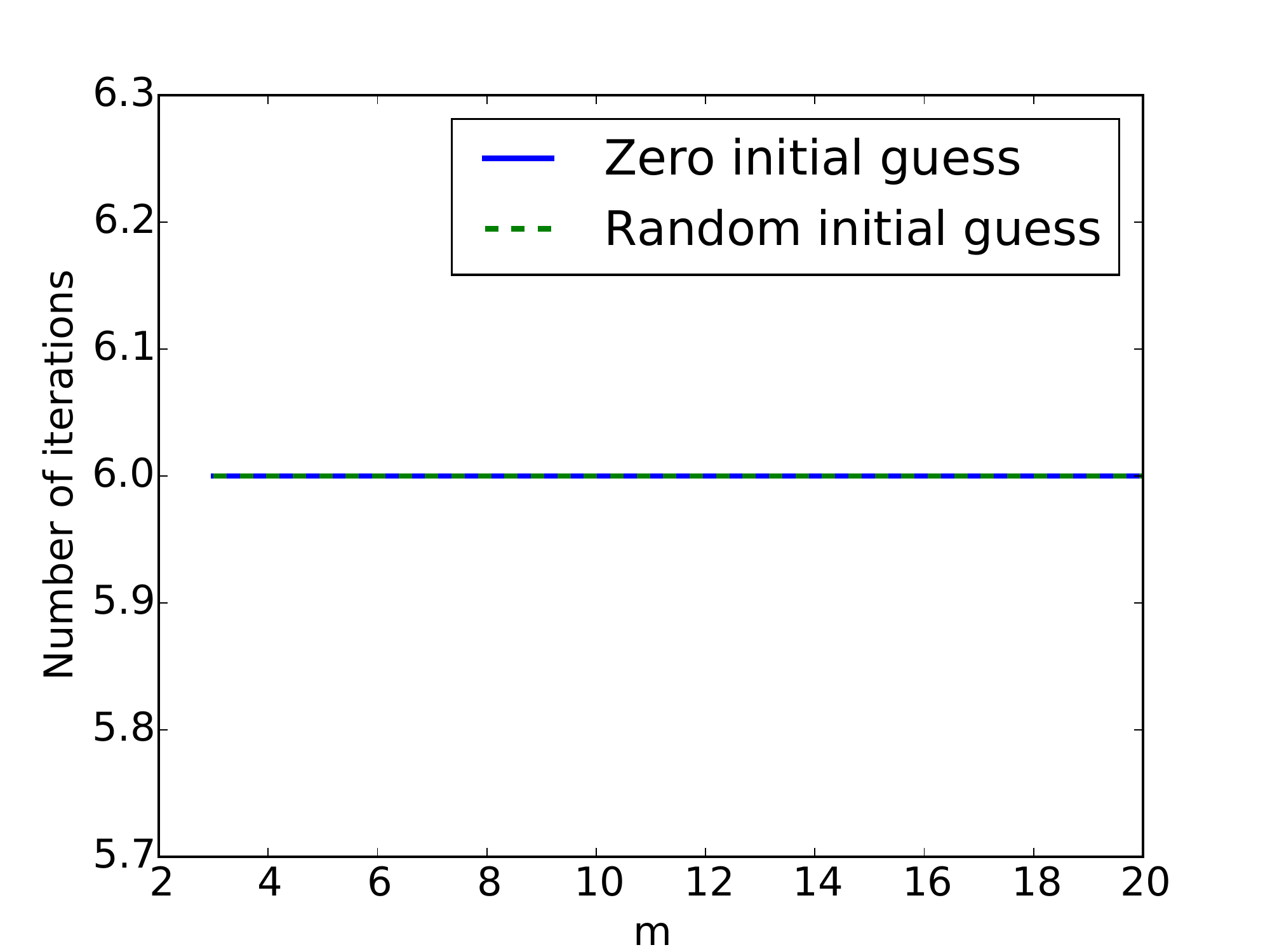}
    \caption{Number of iterations vs. parameter $m$ for $N=2$
      (left) and $N=32$ (right) in the framework of the
      preconditioned algorithm.}
    \label{Influence_m_NL_pd}
  \end{figure}%

  Secondly, we study the influence of $p$ in Robin transmission
  condition. The numbers of iterations are presented in Table
  \ref{Influence_p_NL_2_32} with different $p$ for $N=2$ and
  $N=32$ (here only $p = 5,10,...,50$ are shown). Both the
  classical algorithm and the preconditioned algorithm
  (Cls./Pd.), as well as the different initial vectors (zero or
  random) are considered. As can be seen, for the preconditioned
  algorithm, the number of iterations is almost the same in each
  case. For the classical algorithm, there exists an optimal $p$
  for each case. 
  \begin{table}[!htbp] 
    \centering
    \begin{tabular}{|c|c|c|c|c||c|c|c|c|}
      \hline
      $p$ & \multicolumn{4}{c||}{$N=2$} & \multicolumn{4}{c|}{$N=32$} \\ 
      \hline
          & \multicolumn{2}{c|}{Zero} & \multicolumn{2}{c||}{Random} & \multicolumn{2}{c|}{Zero} & \multicolumn{2}{c|}{Random} \\
      \hline
          & Cls. & Pd. & Cls. & Pd. & Cls. & Pd. & Cls. & Pd.\\
      \hline
      $5$  & 57 & 6 & 548 & 6 & 57 & 6 & 582 & 6 \\
      $10$ & 35 & 6 & 297 & 6 & 35 & 6 & 316 & 6 \\
      $15$ & 32 & 6 & 227 & 6 & 33 & 6 & 241 & 6 \\
      $20$ & 36 & 6 & 200 & 6 & 36 & 6 & 212 & 6 \\
      $25$ & 41 & 6 & 189 & 6 & 41 & 6 & 200 & 6 \\
      $30$ & 46 & 6 & 186 & 6 & 47 & 6 & 200 & 6 \\
      $35$ & 53 & 6 & 188 & 6 & 53 & 6 & 204 & 6 \\
      $40$ & 59 & 6 & 194 & 6 & 60 & 6 & 211 & 6 \\
      $45$ & 66 & 6 & 208 & 6 & 66 & 6 & 223 & 6 \\ 
      $50$ & 73 & 6 & 216 & 6 & 73 & 6 & 234 & 6 \\ 
      \hline
    \end{tabular}
    \caption{\small Number of iterations vs. parameter $p$.}
    \label{Influence_p_NL_2_32}
  \end{table}


  In conclusion, the use of the preconditioner allows to reduce both the number of iterations and the computation time. In addition, the preconditioned algorithm is not sensitive to the transmission conditions as well as the parameters in these transmission conditions.

  \subsection{Simulation of Bose-Einstein condensates}

  In this part, we apply the parallel algorithms to BEC simulation. Before comparing numerically the algorithms and making dynamic simulation of quantized vortex lattices, we recall some facts about BEC.

  \subsubsection{Gross-Pitaevski equation}
  A Bose–Einstein condensate (BEC) is a state of matter of a dilute gas of bosons cooled to temperatures very close to absolute zero. Under such conditions, a large fraction of bosons occupy the lowest quantum state, at which point macroscopic quantum phenomena become apparent. One of the models for BEC is the Gross-Pitaevskii (GPE) equation  \cite{Bao2012bec_review,Bao2005rot,Bao2004groundstate,Antoine2013bec}. In this paper, we consider the GPE equation defined on a bounded spatial domain with the same boundary conditions as \eqref{Sch}:
  \begin{equation}
    \label{GPE}
    \left\{
      \begin{array}{ll}
        i\partial_t u  + \frac{1}{2}\Delta u - V(x,y) u - \beta |u|^2 u + \omega \cdot L_z u = 0, \ (t,x,y)\in (0,T)\times \Omega, \\
        u(0,x,y) = u_0(x,y).
      \end{array} 
    \right.
  \end{equation}
  The constant $\beta$ describes the strength of the short-range two-body interactions (positive for repulsive interaction and negative for attractive interaction) in a condensate. The constant $\omega \in \mathbb{R}$ represents the angular velocity, the $z$-component of the angular momentum $L_z$ is given by
  \begin{displaymath}
    L_z = -i(x\partial_y - y\partial_x).
  \end{displaymath}
  The potential here is 
  \begin{displaymath}
    V(x,y) = \frac{1}{2}( \gamma_x^2 x^2 + \gamma_y^2 y^2),\ \gamma_x, \gamma_y \in \mathbb{R}. 
  \end{displaymath}
  The GPE equation is a type of nonlinear Schr\"{o}dinger equation. One of the difficulties in the simulation of Bose-Einstein condensates derives from the term of rotation.
  Recently, the authors of \cite{Bao2013bec} introduced a coordinate transformation that allows to write the GPE equation in this new coordinates as a nonlinear Schr\"{o}dinger equation \eqref{GPE_Trans} with a time-dependent potential but without the rotation term. Thus, the algorithms that we presented in the previous sections are applicable for GPE equation. For $\forall t \geqslant 0$, the orthogonal rotational matrix $A(t)$ is defined by
  \begin{displaymath}
    A(t) = 
    \begin{pmatrix}
      \cos(\omega t) & \sin(\omega t) \\
      -\sin(\omega t) & \cos(\omega t)
    \end{pmatrix}.
  \end{displaymath}
  The transformed Lagrange coordinate $(\widetilde{x},\widetilde{y})$ is then defined as
  \begin{equation}
    \label{x1Ax}
    \begin{pmatrix}
      \widetilde{x} \\
      \widetilde{y}
    \end{pmatrix}
    = 
    A^{-1}(t)
    \begin{pmatrix}
      x\\
      y
    \end{pmatrix}
    = 
    A^{\top}(t)
    \begin{pmatrix}
      x\\
      y
    \end{pmatrix}.
  \end{equation}
  In this new coordinate, the GPE equation \eqref{GPE} could be written as
  \begin{equation}
    \label{GPE_Trans}
    \left\{
      \begin{array}{ll}
        i\partial_t \widetilde{u}  + \frac{1}{2}\Delta \widetilde{u} - V_t (t,\widetilde{x},\widetilde{y}) \widetilde{u} - \beta |\widetilde{u}|^2 \widetilde{u} = 0, \ t \in (0,T), \\
        \widetilde{u}(0,\widetilde{x},\widetilde{y}) = \widetilde{u}_0(\widetilde{x},\widetilde{y}),
      \end{array} 
    \right.
  \end{equation}
  where
  \begin{equation}
    \label{u1u}
    \widetilde{u}(t,\widetilde{x},\widetilde{y}) := u(t,x,y), \
    V_t(t,\widetilde{x},\widetilde{y}) := V(x,y), \ \text{where}   \ (x,y)^{\top} = A(t) (\widetilde{x},\widetilde{y})^{\top}.
  \end{equation}
  
  Formally, the only difference between the equation
  \eqref{GPE_Trans} and the Schr\"{o}dinger equation \eqref{Sch}
  is the constant in front of the Laplace operator
  $\Delta$. Thus, we could directly apply the domain
  decomposition algorithms to the equation \eqref{GPE_Trans} on
  the spatial domain $\Omega=(x_l,x_r)\times (y_b,y_u)$. {The Robin
    transmission condition and the transmission condition
    $S_{\mathrm{pade}}^m$ are given by \eqref{TCSRobin}, \eqref{TCS}
    and \eqref{TCS_phi}. A minor modification concerns the constant before the operator
    $\Delta_{\Gamma_j}$ in \eqref{TCS_phi} which is $\frac{1}{2}$ here.} 
  
  %

  Once the solution $\widetilde{u}$ is computed numerically, it is possible to reconstruct the solution $u$ by \eqref{u1u}. At time $t$, the computational domain of $\widetilde{u}(t,\widetilde{x},\widetilde{y})$ is $\Omega=(x_l,x_r)\times (y_b,y_u)$ and the computational domain of $u(t,x,y)$ is $A(t)\Omega$ (see figure \ref{AtD}). The domains $A(t)\Omega$ for $t \geqslant 0$ share a common disc. The values of $u(t,x,y)$ within the maximum square (the valid zone) are all in the disc, which could be computed by interpolation. The valid zone is
  \begin{displaymath}
    (\frac{x_l}{\sqrt{2}}, \frac{x_r}{\sqrt{2}}) \times (\frac{y_b}{\sqrt{2}}, \frac{y_u}{\sqrt{2}}).
  \end{displaymath}

  \begin{figure}[!htbp]
    \centering
    \begin{subfigure}{0.35\textwidth}
      \includegraphics[width=\textwidth]{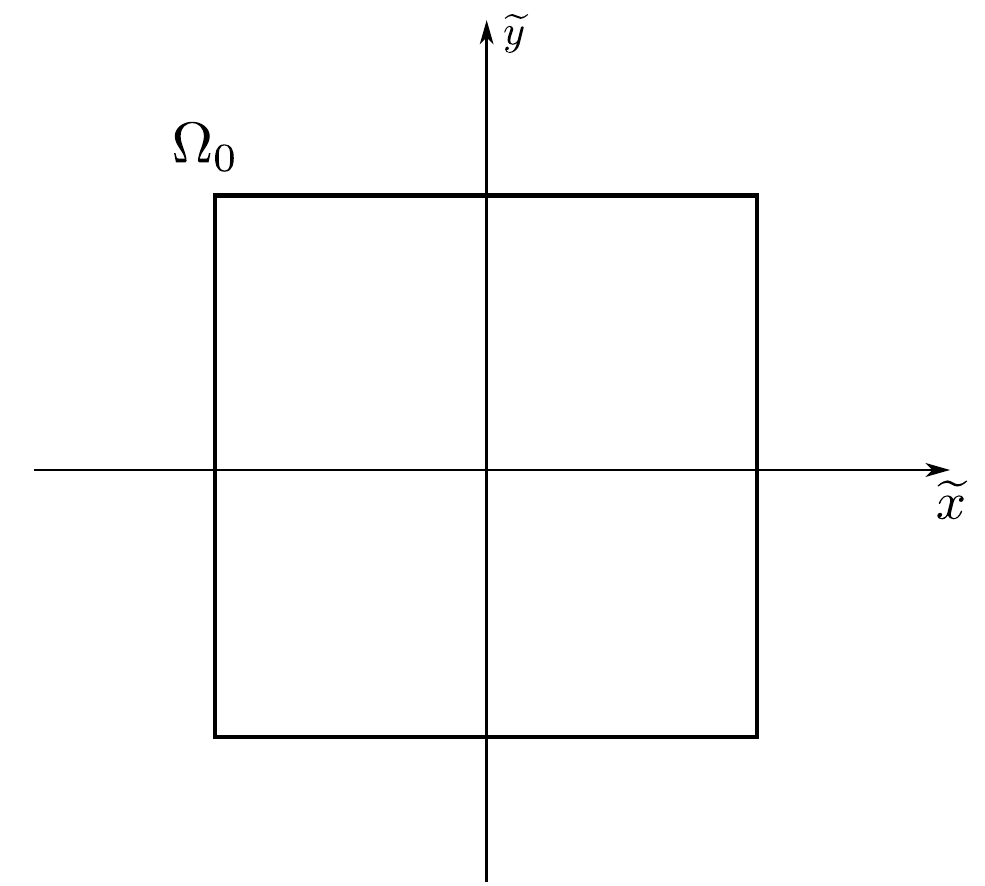}
      \caption{\small $(\widetilde{x},\widetilde{y}) \in \Omega$.}
    \end{subfigure}
    \hspace{0.05\textwidth}
    \begin{subfigure}{0.35\textwidth}
      \includegraphics[width=\textwidth]{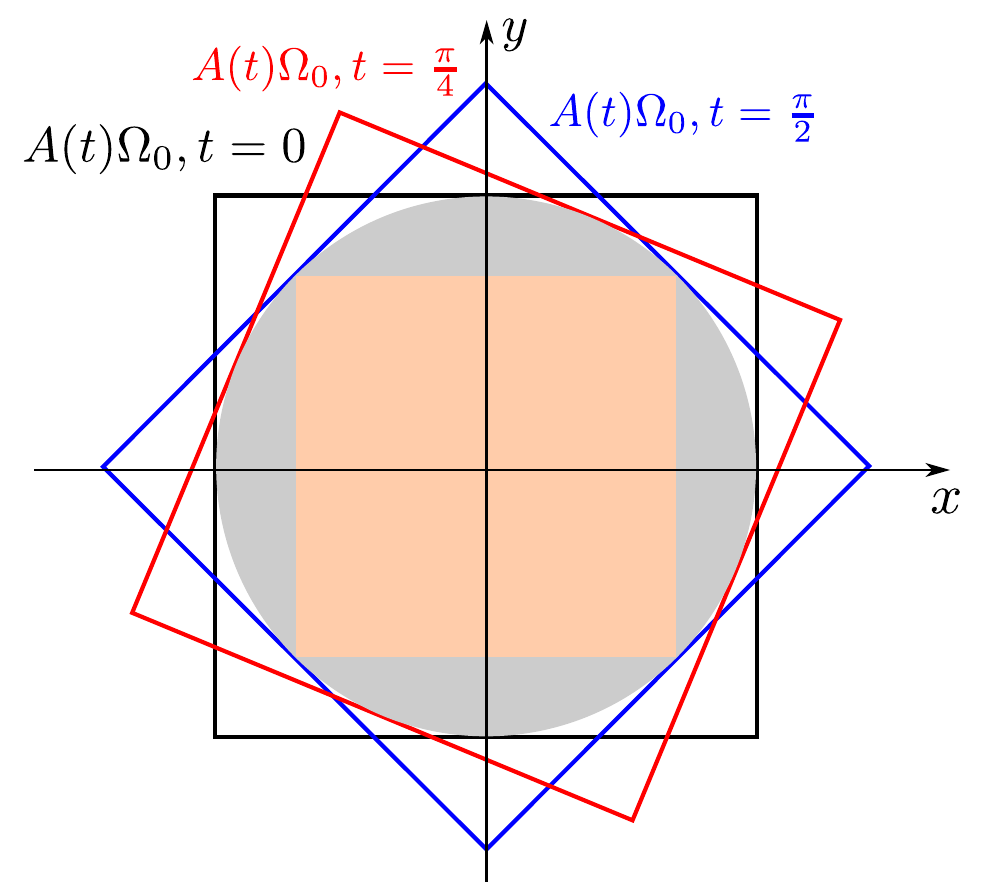}
      \caption{\small $(x,y) \in A(t) \Omega$.}
    \end{subfigure}
    \caption{(a) The computational domain $\Omega$. (b) The domain $A(t) \Omega$ at some different times:  {\color{black}{$t=0$}}, {\color{red}{$t=\pi/4$}} and  {\color{blue}{$t=\pi/2$}} where $\omega=0.5$.}
    \label{AtD}
  \end{figure}

  \subsubsection{Comparison of algorithms}

  In this part, we fix the physical domain to be $\Omega = (-16,16)\times (-16,16)$. The initial datum is taken as a Gaussian 
  \begin{displaymath}
    u_{0}(x,y) = \frac{1}{\pi^{1/4}} e^{\frac{-(x^2+2y^2)}{2}},\ (x,y) \in \mathbb{R}^2,
  \end{displaymath}
  where the coefficients are $\omega=0.4$ and $\beta=10.15$. The
  time step is fixed as $\Delta t = 0.0001$. Firstly, we use a
  wide mesh $\Delta x= \Delta y = 1/32$, which generates $1024
  \times 1024$ unknowns on $\Omega$. It is possible to solve the
  GPE equation \eqref{GPE_Trans} on the complete domain $\Omega$
  under our memory limitation (32G) without using the parallel
  algorithms (classical or preconditioned algorithm). However,
  the computation time could be very long. Thus, we use here a
  small final time $T = 0.1$. Using the same notations as in the
  previous sections, we show in Table \ref{time_dx32} the
  computation times of the two algorithms with Robin and
  $S_{\mathrm{pade}}^m$ transmission conditions. Since the
  boundary condition imposed on $\Omega$ is associated with the
  transmission operator, the reference times $T^{\mathrm{ref}}$
  for the two transmission condition are different. In BEC
  simulation, a small time step is necessary. According to our
  experiments, when a small $\Delta t$ is considered, a large $m$ in $S_{\mathrm{pade}}^m$ transmission
  condition is needed to ensure fast convergence. Thus, the
  use of the transmission condition
  $S_{\mathrm{pade}}^m$ is much more expensive than the
  transmission condition Robin. We can also see that the
  computation times of the classical algorithm
  ($T_{\mathrm{pc}}$) and the preconditioned algorithm
  ($T_{\mathrm{nopc}}$) are scalable.  

  \begin{table}[!htbp]
    \renewcommand{\arraystretch}{1.2}
    \centering
    \begin{tabular}{|c|c|c|c|c|c|c|}
      \hline
      $N$& & 2 & 4 & 8 & 16 & 32 \\
      \hline
      \multirow{3}{*}{Robin, $p=180$}
         & $T^{\mathrm{ref}}$ & \multicolumn{5}{c|}{5.68} \\
      \cline{2-7}
         & $T_{\mathrm{nopc}}$ & 5.68 & 2.66 & 1.28 & 0.68 & 0.33 \\
      \cline{2-7}
         & $T_{\mathrm{pc}}$& 3.49 & 1.60 & 0.77 & 0.44 & 0.24 \\
      \hline
      \multirow{3}{*}{$S_{\mathrm{pade}}^m$, $m=76$} 
         & $T^{\mathrm{ref}}$ & \multicolumn{5}{c|}{8.41} \\
      \cline{2-7}
         & $T_{\mathrm{nopc}}$ & > 20 & 10.70 & 7.40 & 5.07 & 4.23 \\
      \cline{2-7}
         & $T_{\mathrm{pc}}$ & 6.30 & 3.52 & 2.30 & 1.68 & 1.37 \\
      \hline  
    \end{tabular}
    \caption{Computation time in hours with the mesh $\Delta x =\Delta y=1/32$.}
    \label{time_dx32}
  \end{table}

  We make the tests with a finer mesh $\Delta x =1/1024$,
  $\Delta y=1/64$ with the Robin transmission condition 
  {since it has been seen that the implementation with 
  the transmission condition $S_{\mathrm{pade}}^m$ 
  is much more expensive than with the Robin transmission condition in the context of Gross-Pitaevski equation.}
  The complete domain is
  decomposed into $N=128,256,512,1024$ subdomains. The
  computation times are presented in Table \ref{time_dx1024}. We
  could see that the both algorithms are scalable. In addition,
  the preconditioner allows to reduce the total computation
  time. However, since the implementation of the preconditioner
  consumes memory, the memory is not sufficient in the case
  $N=128$. 

  \begin{table}[!htbp]
    \renewcommand{\arraystretch}{1.2}
    \centering
    \begin{tabular}{|c|c|c|c|c|}
      \hline
      $N$& 128 & 256 & 512 & 1024 \\
      \hline
      $T_{\mathrm{nopc}}$, $p=95$ & 19.3 & 8.8 & 5.0 & 2.1 \\
      \hline
      $T_{\mathrm{pc}}$, $p=95$ & * & 2.3 & 1.6 & 0.8\\
      \hline
    \end{tabular}
    \begin{tablenotes}
    \item \small *: the memory is not sufficient.
    \end{tablenotes}
    \caption{Computation time in hours with the mesh $\Delta x =1/1024$, $\Delta y=1/64$.}
    \label{time_dx1024}
  \end{table}


  \subsubsection{Dynamic simulation of quantized vortex lattices}
    {According to the studies in the previous subsection, }
  we apply the algorithms with the Robin transmission condition to study the dynamics of quantized vortex lattices for BEC with rotation. In this simulation, the nonlinear potential and the parameters are
  \begin{displaymath}
    V(x,y) = \frac{1}{2}(x^2 + y^2),\ \beta = 1000,\  \omega = 0.9.
  \end{displaymath}
  The initial solution $u_0$ is a stationary vortex lattice \cite{Bao2012bec_review, Antoine2014gpelab}. The stationary solution $\phi$ of (\ref{GPE}) is defined as 
  \begin{equation}
    \label{uphie}
    u(t,x,y) = \phi(x,y) e^{-i\mu t},
  \end{equation}
  where $\mu$ is the chemical condensation potential. By substituing (\ref{uphie}) in (\ref{GPE}), we have
  \begin{displaymath}
    \mu \phi =  -\frac{1}{2} \phi + V\phi + \beta |\phi|^2 \phi - \omega L_z \phi,
  \end{displaymath} 
  with the constraint of normalisation
  \begin{displaymath}
    || \phi ||^2_2 = \int_{\mathbb{R}^2} |\phi(x,y)|^2 dxdy = 1.
  \end{displaymath}
  This is therefore a nonlinear eigenvalue problem. The eigenvalue $\mu$ can be computed from its corresponding eigenvector $\phi$ by
  \begin{displaymath}
    \mu_{\beta,\omega}(\phi) = E_{\beta,\omega}(\phi) + \frac{\beta}{4} \int_{\mathbb{R}^2} |\phi(x,y)|^4 dxdy,
  \end{displaymath}
  where 
  \begin{equation}
    E_{\beta,\omega}(\phi) = \frac{1}{2} \int_{\mathbb{R}^2} ( |\nabla \phi|^2 + V|\phi|^2 + \beta |\phi|^4 - \omega \overline{\phi} L_z \phi ) dxdy.
  \end{equation}
  The ground state of a BEC is defined as the solution of minimization problem, denoted by $\phi_g$,
  \begin{displaymath}
    E_{\beta,\omega} (\phi_g) = \min_{\phi \in S} E_{\beta,\omega} (\phi),
  \end{displaymath}
  where $S=\{ \phi |\  ||\phi||_2 =1, E_{\beta,\omega} <\infty \}$.
  
  For our simulation, we take the solution of minimization problem as the datum initial
  \begin{equation}
    \label{u0phig}
    u_0(x,y) = \phi_g(x,y).
  \end{equation}
  It is computed by BESP method (Backward Euler Sine
  Pseudospectral) \cite{Bao2006besp} using GPELab
  \cite{Antoine2014gpelab}, a matlab toolbox developed for the
  computation of the ground states and the dynamics of quantum
  systems modeled by GPE equations.

  The complete domain $\Omega= (-16,16) \times (-16,16)$ is decomposed into $N=32$ subdomains. We fix the time step as $\Delta t=0.0001$. The mesh is $\Delta x= \Delta y=1/32$. The parameter $p$ here is $p=180$. Figure \ref{simulation1} shows the contours of the solution $|u(t,x,y)|^2$ at some different times. The solution is illustrated in the valid zone $(-16/\sqrt{2},16/\sqrt{2}) \times (-16/\sqrt{2},(16-\Delta y)/\sqrt{2})$. The total computation time is about 16 hours.

  \begin{figure}[!htbp]
    \centering
    \includegraphics[width=0.35\textwidth]{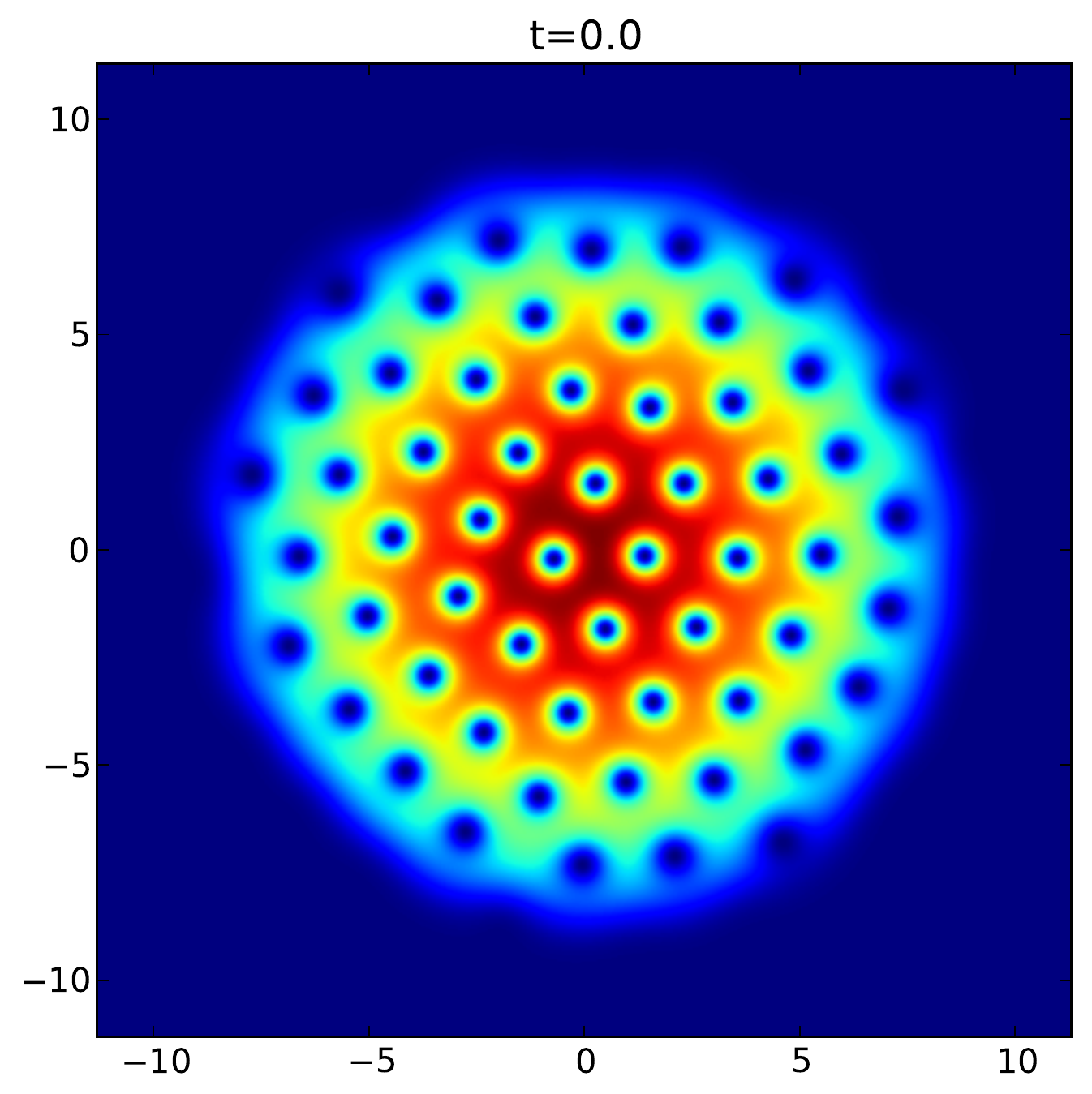}
    \hspace{0.1\textwidth}
    \includegraphics[width=0.35\textwidth]{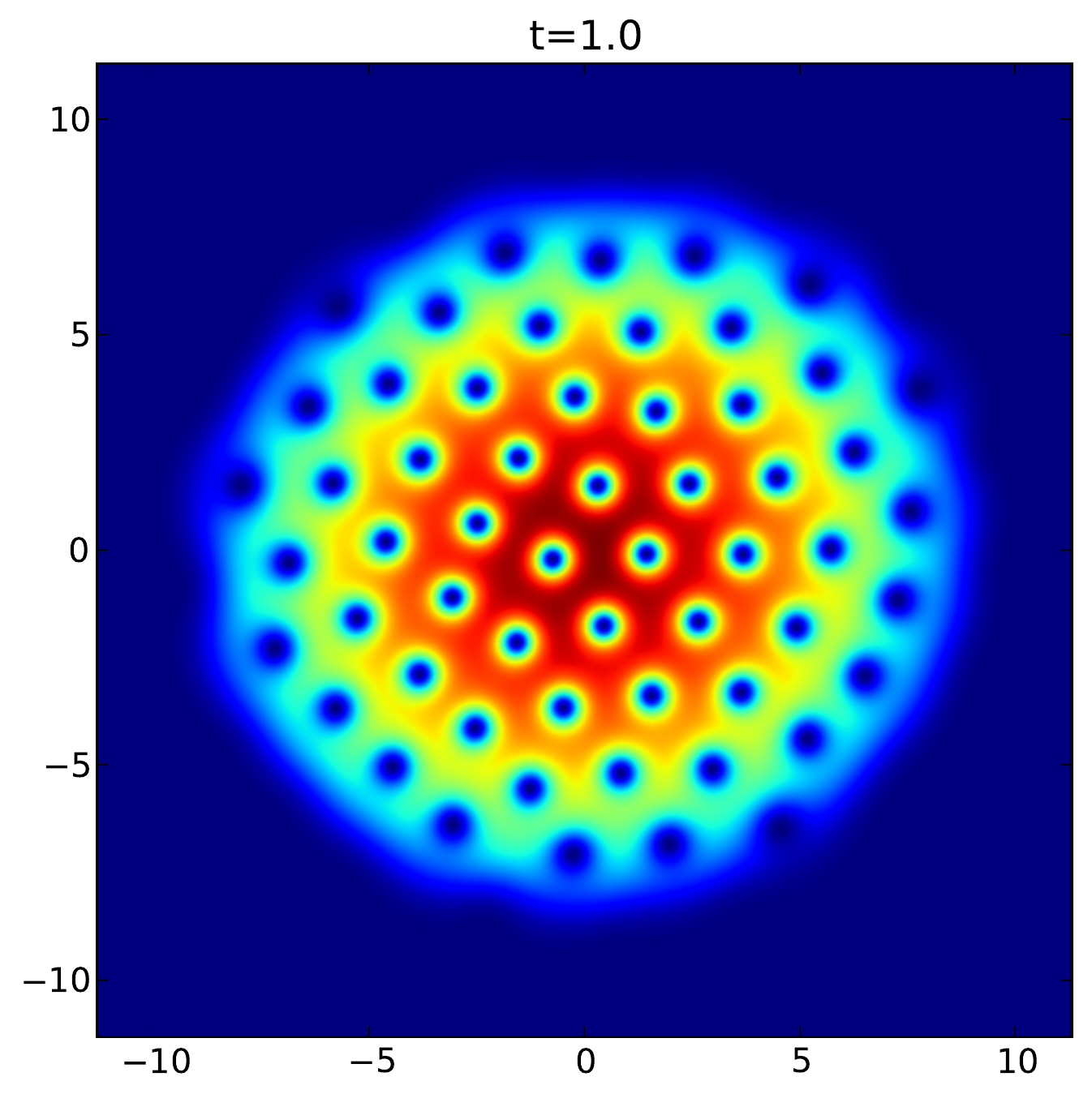}
    \includegraphics[width=0.35\textwidth]{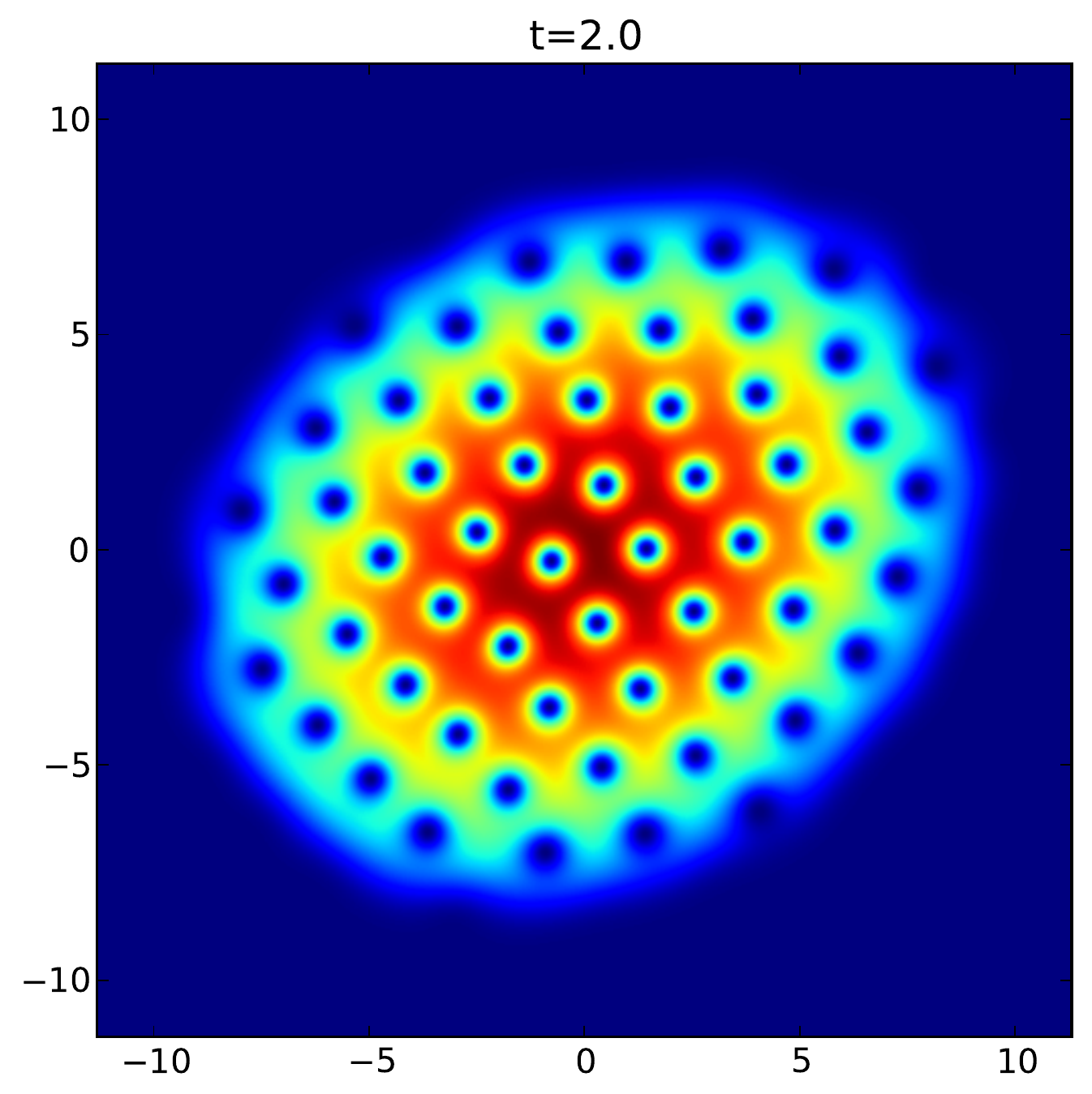}
    \hspace{0.1\textwidth}
    \includegraphics[width=0.35\textwidth]{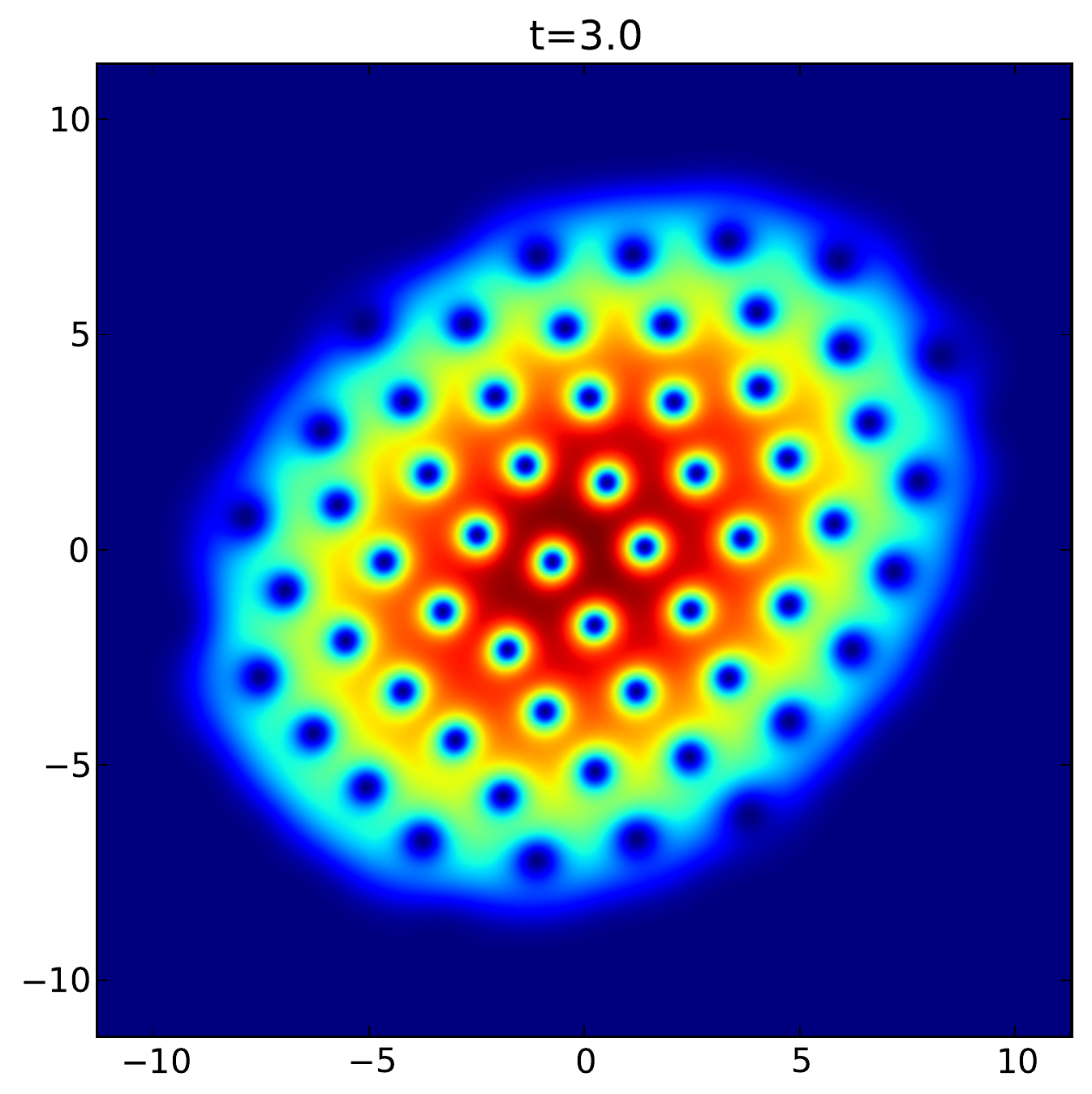}
    \includegraphics[width=0.35\textwidth]{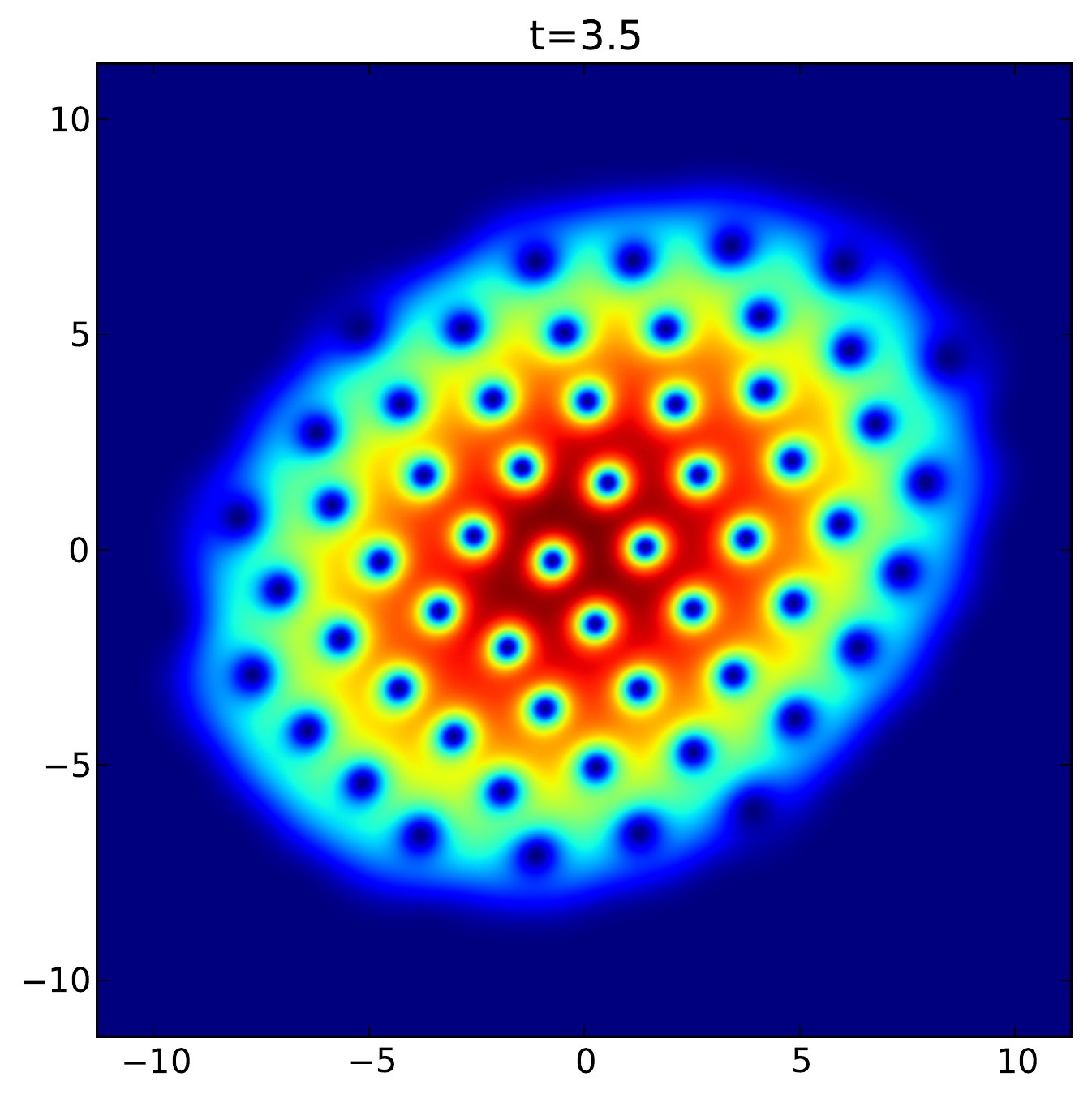}
    \hspace{0.1\textwidth}
    \includegraphics[width=0.35\textwidth]{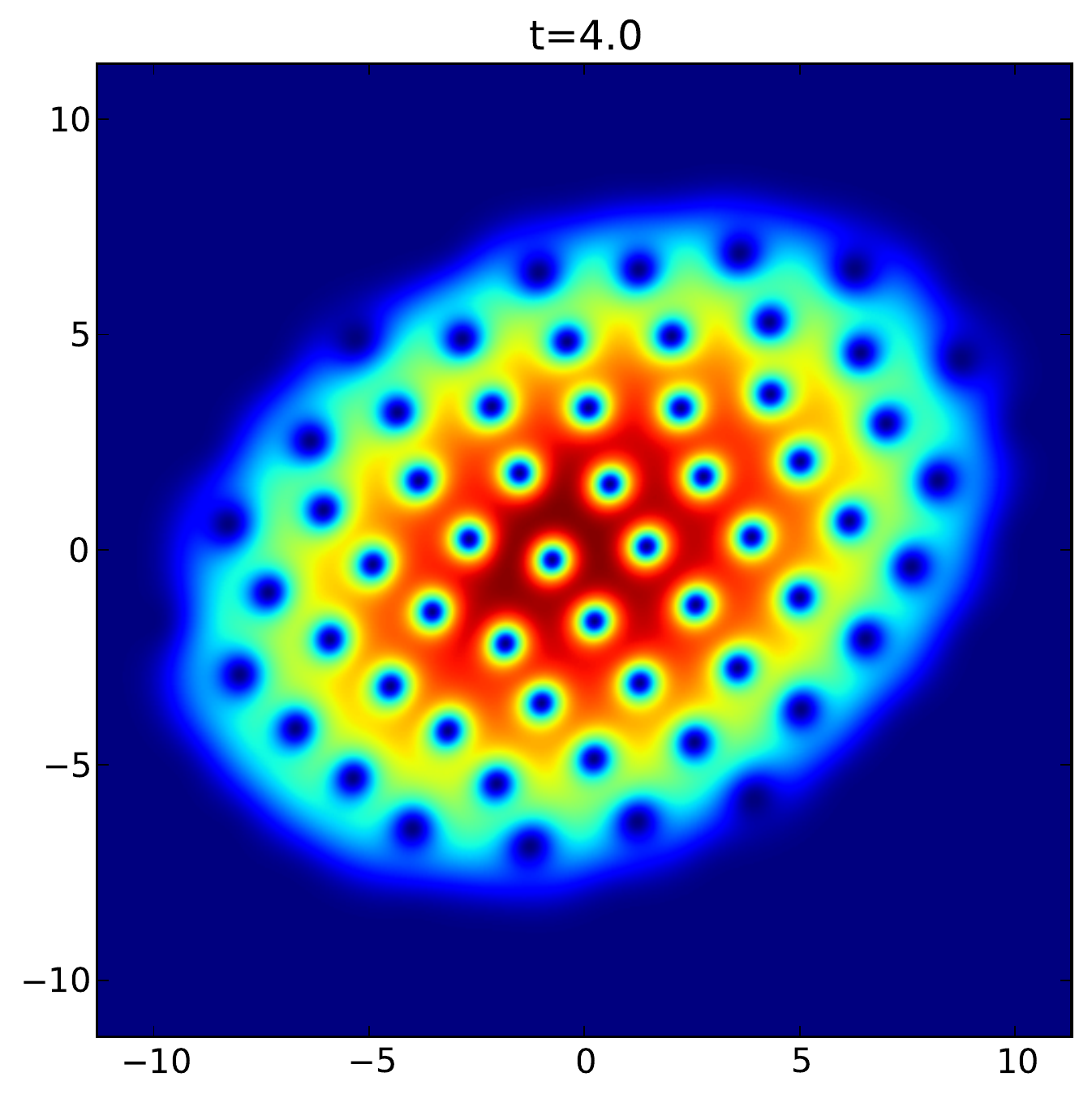}
    \caption{Contours of solution $|u(t,x,y)|^2$ at some different times.}
    \label{simulation1}
  \end{figure}

  \section{Conclusion and perspective}

  We applied the optimized Schwarz method to the two dimensional
  nonlinear Schr\"{o}dinger equation and GPE equation. We
  proposed a preconditioned algorithm which allows to reduce the
  number of iterations and the computation time. According to
  the numerical tests, the preconditioned algorithm is not
  sensitive to the transmission conditions (Robin,
  $S_{\mathrm{pade}}^m$) and the parameters in these
  conditions. In addition, the parallel algorithms are applied
  to the BEC simulation. We can obtain an accurate solution by
  using the parallel algorithms and the computation time of the
  preconditioned algorithm is less than the classical one. 

{
  One perspective could be to use a partially constructed 
  $I-\mathcal{L}_h$ as the preconditioner in the context of 
  the multilevel preconditioner. The construction and the
  implementation should be less expensive.
}

  \section*{Acknowledgements} We acknowledge Pierre Kestener
  (Maison de la Simulation Saclay France) for the discussions
  about the parallel programming. This work was partially
  supported by the French ANR grant ANR-12-MONU-0007-02 BECASIM
  (Mod\`eles Num\'eriques call). The
  first author also acknowledges support from the French ANR grant BonD
  ANR-13-BS01-0009-01.  

  \bibliographystyle{abbrv}
  \bibliography{Bib2}

\end{document}